\documentclass[10pt,french]{smfart}

\usepackage[T1]{fontenc}
\usepackage[english,french]{babel}
\usepackage{latexsym,amscd,color}
\usepackage{amsmath,amsfonts,amssymb,mathrsfs}
\usepackage{enumerate,euscript}
\usepackage{amssymb,url,xspace,smfthm}
\input xy
\xyoption{all}

\DeclareMathOperator{\spm}{Spm}
\DeclareMathOperator{\res}{Res}

\DeclareMathOperator{\gal}{Gal}

\newcommand{\BibTeX}{{\scshape Bib}\kern-.08em\TeX}

\newcommand{\T}{\S\kern .15em\relax }
\newcommand{\AMS}{$\mathcal{A}$\kern-.1667em\lower.5ex\hbox
        {$\mathcal{M}$}\kern-.125em$\mathcal{S}$}
\newcommand{\resp}{\textit{resp}.\xspace}

\DeclareMathOperator{\proj}{Proj}

\DeclareMathOperator{\rg}{rg}

\DeclareMathOperator{\spec}{Spec}

\renewcommand{\P}{\mathbb{P}}

\newcommand{\C}{\mathbb{C}}
\newcommand{\Q}{\mathbb{Q}}

\newcommand{\adeg}{\widehat{\deg}}

\newcommand{\p}{\mathfrak{p}}
\DeclareMathOperator{\sym}{Sym}

\newcommand{\sE}{\mathcal{E}}

\renewcommand{\O}{\mathcal{O}}
\DeclareMathOperator{\Aut}{Aut}
\renewcommand{\L}{\overline{L}}

\newcommand{\f}{\mathbb{F}}
\newcommand{\ndot}{\raisebox{.4ex}{.}}

\title{Fibres non r\'eduites d'un sch\'ema arithm\'etique}
\alttitle{Non-reduced fibers of an arithmetic scheme}
\date{\today}
\date{\today}
\author{Chunhui Liu}
\address{Institute for Advanced Study in Mathematics\\
Harbin Institute of Technology\\
150001 Harbin\\P. R. China}
\email{chunhui.liu@hit.edu.cn}
\begin{document}
\def\smfbyname{}
\begin{abstract}
Pour un sch\'ema r\'eduit projectif sur l'anneau des entiers d'un corps de nombres, l'ensemble des places au dessus desquelles les fibres du sch\'ema ne sont pas r\'eduite est un ensemble fini. On donne une majoration explicite du produit des normes de ces places. Pour cela, on introduit une g\'en\'eralisation de la notion de hauteur sur l'anneau ad\'elique. En utilisant la th\'eorie des vari\'et\'es de Chow, on ram\`ene le cas g\'en\'eral d'un sch\'ema de dimension pure \`a celui d'une hypersurface et on traite ce dernier \`a l'aide du r\'esultant de l'\'equation de l'hypersurface et des d\'eriv\'ees partielles de cette \'equation.

\end{abstract}

\begin{altabstract}
For a reduced projective scheme over the ring of integers of a number field, the set of places over which the fibres of the scheme are not reduced is a finite set. We give an explicit upper bound for the product of the norms of places in this set. For this purpose, we introduce a generalization of the notion of height over the adelic ring. We reduce the general case of a scheme of pure dimension to the case of a hypersurface by using the theory of Chow varieties. The case of a hypersurface is then treated with the help of the resultant of the equation of the hypersurface with some partial derivatives of the equation.

\end{altabstract}
\maketitle

\tableofcontents
%\chapter{Comptage de multiplicit\'e dans une hypersurface}
\section{Introduction}
Soit $X\rightarrow\spec\O_K$ un sch\'ema noeth\'erien r\'eduit, o\`u $K$ est un corps de nombres et $\O_K$ est l'anneau des entiers de $K$. On d\'esigne par $\spm \O_K$ l'ensemble des id\'eaux maximaux de l'anneau $\O_K$. Une place $\p\in\spm\O_K$ est appel\'ee \textit{place non r\'eduite} du sch\'ema $X\rightarrow\spec\O_K$ si la fibre $X_{\f_\p}=X\times_{\spec\O_K}\spec\f_\p\rightarrow\spec\f_\p$ n'est pas r\'eduite, o\`u $\f_\p$ est le corps r\'esiduel de $\O_K$ par rapport \`a $\p$. D'apr\`es \cite[Th\'eor\`eme (9.7.7)]{EGAIV_3}, il n'y a qu'un nombre fini d'id\'eaux maximaux $\p\in \spm\O_K$ tels que la fibre $X_{\f_\p}\rightarrow\spec\f_\p$ ne soit pas r\'eduite.

Il est naturel de consid\'erer une description num\'erique des places non r\'eduites. Par exemple, on consid\`ere la majoration du nombre de ces id\'eaux maximaux ou la majoration du produit des normes de ces id\'eaux maximaux.

Ern\'e a consid\'er\'e un sujet similaire. Dans \cite{Erne_hypersurface}, \'etant donn\'ee une hypersurface projective g\'eom\'etriquement int\`egre d'un degr\'e fix\'e, par le th\'eor\`eme arithm\'etique de B\'ezout introduit dans \cite[Theorem 5.4.4, Theorem 5.5.1]{BGS94}, elle \'etudie la majoration du produit des normes de id\'eaux maximaux tels que les fibres contiennent une hypersurface d'un autre degr\'e plus petit fix\'e. Dans \cite{Erne_general}, elle \'etudie le cas de sch\'ema projectif g\'eom\'etriquement int\`egre en utilisant la th\'eorie des vari\'et\'es de Chow.
\subsection{R\'esultat principal}
Dans cet article, pour un sch\'ema projectif r\'eduit sur un corps de nombres arbitraire, on donnera une majoration du produit des normes des id\'eaux maximaux non r\'eduits.
\begin{theo}[Th\'eor\`eme \ref{reduced default}]
  Soit $X$ un sous-sch\'ema ferm\'e r\'eduit de dimension pure $d$ et de degr\'e $\delta$ de $\mathbb P^n_K$, dont l'adh\'erence sch\'ematique dans $\mathbb P^n_{\O_K}$ est $\mathscr X$. On d\'esigne par $\mathcal Q(\mathscr X)$ l'ensemble des places sur lesquelles les fibres de $\mathscr X$ ne sont pas r\'eduites. Alors on a
  \[\frac{1}{[K:\Q]}\sum_{\p\in\mathcal Q(\mathscr X)}\log N(\p)\leqslant(2\delta-1)h_{\overline{\O(1)}}(X)+C_0(d,n,\delta),\]
  o\`u $N(\p)=\#(\O_K/\p)$, et $h_{\overline{\O(1)}}(X)$ est une hauteur de $X$ (voir la d\'efinition \ref{arakelov height of projective variety}). De plus, on explicitera la constant $C_0(d,n,\delta)$ dans le th\'eor\`eme \ref{reduced default}.
\end{theo}
\subsection{M\'ethode}
D'abord on r\'esout ce probl\`eme dans le cas o\`u $X$ est une hypersurface projective (voir le th\'eor\`eme \ref{default reduissant}). On consid\`ere le polyn\^ome homog\`ene qui d\'efinit $X$ comme un polyn\^ome \`a coefficients dans l'anneau ad\'elique $\mathbb A_K$ (voir la remarque \ref{compare adelic and classic height}). Dans ce cas-l\`a, on consid\`ere son polyn\^ome primitif au sens ad\'elique. De plus, on consid\`ere un r\'esultant de ce polyn\^ome primitif, qui est non-nul lorsque l'hypersurface $X$ est r\'eduite. On donnera une majoration des id\'eaux maximaux tels que les r\'eductions du polyn\^ome modulo lequels s'annulent, ce qui donne un contr\^ole des places non r\'eduites.

En suite, on r\'esout le cas o\`u $X$ est un sch\'ema projectif de dimension pure en utilisant la th\'eorie des vari\'et\'es de Chow et des vari\'et\'es de Cayley (voir le th\'eor\`eme \ref{reduced default}). Si un sch\'ema de dimension pure est r\'eduit, toute composante irr\'eductible de sa forme de Chow ou de sa vari\'et\'e de Cayley est de multiplicit\'e $1$ (voir la d\'efinition \ref{fundamental cycle}).

La m\'ethode dans cet article est diff\'erente de celle dans \cite{Erne_hypersurface,Erne_general} et permet d'obtenir des r\'esultats explicits. Compar\'ee aux estimations dans \cite{Erne_hypersurface,Erne_general}, notre majoration donne une meilleure d\'ependance de la hauteur de $X$ et, globalement, de meilleures constantes. Comme on utilise une m\'ethode explicite, il faut utiliser la hauteur classique (voir la d\'efinition \ref{classical height of hypersurface}) directement. Pour le cas g\'en\'eral de dimension pure, on a besoin de comparer certaines hauteurs de $X$.
\subsection{Organisation de l'article}
Cet article est organis\'e comme suivant. Dans la section 2, on rappellera la th\'eorie des vari\'et\'es de Chow et des vari\'et\'es de Cayley, qui est purement g\'eom\'etrique. Dans la section 3, on comparera quelques hauteurs d'un sch\'ema arithm\'etique de dimension pure. Dans la section 4, on construira un r\'esultant particulier, et donnera une majoration des id\'eaux maximaux de $\O_K$ tels que les r\'eductions de ce r\'esultant modulo lesquels  sur laquelles s'annulent. Dans la section 5, on donnera la majoration mentionn\'ee plus haut pour le cas d'une hypersurface dans le th\'eor\`eme \ref{default reduissant}, o\`u l'on consid\`ere le r\'esultant de l'\'equation qui d\'efinit l'hypersurface. Dans la section 6, on traitera le cas d'un sch\'ema de dimension pure g\'en\'eral par la th\'eorie des vari\'et\'es de Chow et des vari\'et\'es de Cayley au th\'eor\`eme \ref{reduced default}.
\subsection*{Remerciements}
Ce travail fait partie de la th\`ese de l'auteur pr\'epar\'ee \`a l'Universit\'e Paris Diderot - Paris 7. L'auteur voudrait remercier profond\'ement ses directeurs de th\`ese Huayi Chen et Marc Hindry pour leurs sugg\`estions autour de ce travail. De plus, l'auteur voudrait remercier le rapporteur anonyme pour sa lecture attentive et ses nombreuses suggestions qui ont permis d'am\'eliorer grandement le pr\'esent texte.
\section{La vari\'et\'e de Chow et la vari\'et\'e de Cayley}\label{chow form and cayley form}

 Dans cette section, on rappellera les notions de vari\'et\'e de Chow et de vari\'et\'e de Cayley d'un sch\'ema projectif de dimension pure. Dans tout l'article, les anneaux consid\'er\'es sont commutatifs et unitaires sauf mention contraire.

\subsection{La formation sur un corps}\label{Cayley form over field}
Pour la construction des vari\'et\'es de Chow et des vari\'et\'es de Cayley sur un corps, on utilise une approche inspir\'ee par \cite[\S3.1]{Chen1}. On renvoie \`a \cite{Gelfandal94} pour une introduction syst\'ematique de cette th\'eorie.

f
Soient $A$ un anneau, et $M$ un $A$-module. On d\'esigne par $\ell_A(M)$ la \textit{longueur} de $M$ comme un $A$-module. On revoie les lecteurs \`a \cite[\S2.4]{GTM150} pour plus de d\'etails.

Soit $m$ un entier positif. On d\'esigne par $\sym_A^m(M)$ le $m$-i\`eme produit sym\'etrique de $M$, ou par $\sym^m(M)$ s'il n'y a pas d'ambigu\"it\'e sur $A$. De plus, on d\'esigne par
\[\sym_A(M):=\bigoplus_{i\in\mathbb N}\sym_A^i(M).\]

La notion introduite au-dessous provient de \cite[\S 1.5]{Fulton}.
\begin{defi}\label{fundamental cycle}
Soient $X$ un sch\'ema noeth\'erien de dimension pure, et $\mathcal C(X)$ l'ensemble des composantes irr\'eductibles de $X$. On d\'efinit le \textit{cycle fondamental} de $X$ comme la somme formelle
\begin{equation*}
  [X]=\sum_{X'\in\mathcal C(X)}\ell_{\O_{X,X'}}(\O_{X,X'})X'.
\end{equation*}
 De plus, l'entier $\ell_{\O_{X,X'}}(\O_{X,X'})$ est appel\'e la \textit{multiplicit\'e} de la composante irr\'eductible $X'\in\mathcal C(X)$ dans $X$.
\end{defi}
Soient $V$ un espace vectoriel de rang fini sur un corps $k$, et $\mathbb P(V)$ l'espace projectif associ\'e \`a $V$. Dans la d\'efinition \ref{fundamental cycle}, si $X$ est un sous-sch\'ema ferm\'e de dimension pure de $\mathbb P(V)$, on d\'efinit le \textit{degr\'e} du cycle $[X]$ comme
 \[\sum\limits_{X'\in\mathcal C(X)}\ell_{\O_{X,X'}}(\O_{X,X'})\deg_{\O_V(1)}(X'),\]
 qui est \'egal \`a $\deg_{\O_V(1)}(X)$. Dans la suite, on d\'esgine par $\deg(X)$ le degr\'e $\deg_{\O_V(1)}(X)$ pour simplifier les notations.

 Maintenant on donne la construction pr\'ecise de la \textit{vari\'et\'e de Cayley} d'un sous-sch\'ema ferm\'e de dimension pure $X$ de $\mathbb P(V)$. Certaines id\'ees proviennent de \cite[\S3.2.B]{Gelfandal94}, o\`u l'on consid\`ere un plongement dans la grassmannienne. La vari\'et\'e de Caylay param\'etrise les sous-sch\'emas lin\'eaires de dimension $n-\dim(X)-1$ de $\mathbb P(V)$ dont l'intersection avec $X$ est non vide.

 Soit $\check{G}=\hbox{Gr}\left(d+1,V^\vee\right)$ la grassmannienne qui classifie les quotients de rang $d+1$ de $V^\vee$ (ou encore les sous-espaces de rang $d+1$ de $V$), o\`u $V^\vee$ est l'espace dual de $V$. Par le plongement de Pl\"{u}cker $\check{G}\rightarrow\mathbb P\left(\bigwedge^{d+1}V^\vee\right)$, l'alg\`ebre de coordonn\'ees $B(\check{G})=\bigoplus\limits_{D\geqslant 0}B_D(\check{G})$ de $\check{G}$ est une alg\`ebre quotient homog\`ene de l'alg\`ebre $\bigoplus\limits_{D\geqslant 0}\sym^D\left(\bigwedge^{d+1}V^\vee\right)$. Pour expliquer le r\^ole de la coordonn\'ee de Pl\"{u}cker, on consid\`ere la construction suivante: on d\'esigne par
\begin{equation}\label{plucker embedding}
  \theta:V^\vee\otimes_k\left(\bigwedge\nolimits^{d+1}V\right)\rightarrow\bigwedge\nolimits^dV
\end{equation}
 l'homomorphisme qui envoie $\xi\otimes(x_0\wedge\cdots\wedge x_d)$ sur
\begin{equation*}
  \sum_{i=0}^d(-1)^i\xi(x_i)x_0\wedge\cdots\wedge x_{i-1}\wedge x_{i+1}\wedge\cdots\wedge x_d.
\end{equation*}
Soit $\widetilde{\Gamma}$ la sous-vari\'et\'e de $\mathbb P(V)\times_k \mathbb P\left(\bigwedge^{d+1}V^\vee\right)$ qui classifie les point $(\xi,\alpha)$ tels que $\theta(\xi\otimes\alpha)=0$.

Soient \[p': \mathbb P(V)\times_k\mathbb P\left(\bigwedge\nolimits^{d+1}V^\vee\right)\rightarrow\mathbb P(V)\] et \[q': \mathbb P(V)\times_k\mathbb P\left(\bigwedge\nolimits^{d+1}V^\vee\right)\rightarrow\mathbb P\left(\bigwedge\nolimits^{d+1}V^\vee\right)\] les deux projections canoniques, et $v:\widetilde{\Gamma}\rightarrow\mathbb P(V)\times_k\mathbb P\left(\bigwedge^{d+1}V^\vee\right)$ le plongement canonique. On d\'efinit \[p=p'\circ v:\widetilde{\Gamma}\rightarrow\mathbb P(V)\hbox{ et } q=q'\circ v:\widetilde{\Gamma}\rightarrow\mathbb P\left(\bigwedge\nolimits^{d+1}V^\vee\right).\] Alors on a un diagram commutatif comme suit
\[\xymatrix{ & \widetilde{\Gamma}\ar@/^2pc/[ddr]^{q}\ar@/_2pc/[ddl]_{p}\ar[d]^v&\\ & \mathbb P(V)\times_k\mathbb P\left(\bigwedge\nolimits^{d+1}V^\vee\right)\ar[dr]^{q'}\ar[dl]_{p'}&\\ \mathbb P(V)\ar[r]& \spec k&\mathbb P\left(\bigwedge\nolimits^{d+1}V^\vee\right).\ar[l]}\]

Avce la construction au-dessus, on a la proposition suivante, qui est une g\'en\'eralisation de \cite[Proposition 3.4]{Chen1}.
 \begin{prop}\label{cayleyform}
  Soit $X$ un sous-sch\'ema ferm\'e de dimension pure de $\mathbb P(V)$, qui est de dimension $d$. On suppose que $[X]=\sum\limits_{i\in I}m_iX_i$ est le cycle fondamental de $X$. Alors $q_*\left(p^{*}[X]\right)$ est un diviseur sur $\mathbb P\left(\bigwedge^{d+1}V^\vee\right)$. De plus, ce diviseur est de la forme de $\sum\limits_{i\in I}m_i\widetilde{X}'_i$, o\`u chaque $\widetilde{X}'_i$ est une hypersurface int\`egre de degr\'e $\deg(X_i)$ de $\mathbb P\left(\bigwedge^{d+1}V^\vee\right)$, et les $\widetilde{X}'_i$ sont distincts.
\end{prop}
 \begin{proof}
D'abord, on consid\`ere le cas o\`u $X$ est un sch\'ema int\`egre. Dans ce cas-l\`a, la d\'emonstration suivante est inspir\'ee par \cite[Proposition 3.3]{Chen1}. La sous-vari\'et\'e d'incidence $\widetilde \Gamma$ est une fibration sur $\mathbb P (V)$. Comme $X$ est int\`egre, l'inverse sch\'ematique $p^{-1}(X)$ est irr\'eductible. On d\'esigne par $Y=p^{-1}(X)$, consid\'er\'e comme un sous-sch\'ema ferm\'e et int\`egre de $\widetilde \Gamma$. La projection $q$ est propre, donc l'image $Z=q(Y)$ est un sous-sch\'ema ferm\'e int\`egre de $\mathbb P\left(\bigwedge^{d+1}V^\vee\right)$.

Soit $\xi=\spec K$ un point ferm\'e arbitraire de $Z$, qui est correspondant au $(d+1)$-i\`eme puissance ext\'erieure d'un sous-espace de rang $d+1$ de $V$. La fibre $Y_\xi$ co\"incide avec le sous-sch\'ema de $X_{K}$ d\'efini par l'annulation sur $V$ au sens de \eqref{plucker embedding} en prolongant $k$ dans $K$. On prend garde que la dimension de $X_K$ est $d$. Donc $q$ envoie $Y$ dans $Z$ birationnellement et on a $\dim(Z)=\dim(Y)=\dim\left(\mathbb P\left(\bigwedge^{d+1}V^\vee\right)\right)-1$.

Afin de calculer le degr\'e de $Z$ dans $\mathbb P\left(\bigwedge^{d+1}V^\vee\right)$, on consid\`ere l'\'egalit\'e de la classe des cycles
\[[Z]=q_*(p^*[X])=\deg(X)\cdot q_*(p^*[U]),\]
o\`u $U$ est l'espace projectif associ\'e \`a un espace quotient de dimension $d+1$ de $V$ arbitraire prolong\'e dans $\mathbb P\left(\bigwedge^{d+1}V^\vee\right)$ par le plongement de Pl\"ucker. On prend garde que $q_*(p^*[U])$ est la classe premi\`ere de Schubert dans $\mathbb P\left(\bigwedge^{d+1}V^\vee\right)$ (cf. \cite[\S14.7]{Fulton}), alors le degr\'e de $Z$ est $\deg(X)$.

 Dans la suite, on consid\`ere le cas de sch\'emas g\'en\'eraux de dimension pure. Si $X$ est un sous-sch\'ema ferm\'e de dimension pure de $\mathbb P(V)$ avec le cycle fondamental dans l'\'enonc\'e, on a
  \[q_*(p^*[X])=\sum_{i\in I}m_iq_*(p^*[X_i]).\]

Soient $X_i$ et $X_j$ deux composantes irr\'eductibles distinctes de $X$, qui sont consid\'er\'ees comme deux sch\'emas int\`egres. Donc il existe un $\overline k$-point $P$ dans $\mathbb P(V)$, tel que $P\in X_i(\overline k)$ mais $P\not\in X_j(\overline k)$. De plus, on obtient qu'il existe un sous-sch\'ema $\overline k$-lin\'eaire ferm\'e qui intersecte $X_i$ en un sous-sch\'ema non vide mais n'intersecte pas $X_j$. On en d\'eduit $q_*(p^*[X_i])\neq q_*(p^*[X_j])$ comme des cycles premiers.
\end{proof}
 \begin{defi}\label{definition_of_Cayley_vareity}
  Soit $X$ un sous-sch\'ema ferm\'e de dimension pure de $\mathbb P(V)$. On dit que le cycle d\'etermin\'e dans la proposition \ref{cayleyform} est \textit{diviseur de Cayley} de $X$. De plus, on dit que l'hypersurface de $\mathbb P\left(\bigwedge^{d+1}V^\vee\right)$ dont le cycle fondamental est celui d\'etermin\'e dans la proposition \ref{cayleyform} est la \textit{vari\'et\'e de Cayley} de $X$.
\end{defi}
\begin{rema}
  Avec la construction dans la proposition \ref{cayleyform}. Si on remplace $\mathbb P\left(\bigwedge^{d+1}V^\vee\right)$ par la grassmannienne $\hbox{Gr}\left(d+1,V^\vee\right)$, on a presque le même r\'esultat. Dans ce cas-l\`a, le cycle sur $\hbox{Gr}\left(d+1,V^\vee\right)$ obtenu par le sens similaire est appel\'e le \textit{diviseur de Chow}, et l'hypersurface de $\hbox{Gr}\left(d+1,V^\vee\right)$ obtenue par le m\^eme sens que celui dans la d\'efinition \ref{definition_of_Cayley_vareity} est appel\'ee la \textit{vari\'et\'e de Chow}, qui est de degr\'e $\delta$ aussi. Voir \cite[\S3.2.B]{Gelfandal94} pour plus d\'etails sur l'approche.
\end{rema}

\subsection{Le diviseur de Cayley sur un anneau de Dedekind}\label{formation of Cayley form over dedekind ring}
Dans cette partie, on donnera une construction de diviseurs de Cayley sur un anneau de Dedekind. Certaines id\'ees sont inspir\'ees par \cite[\S 4.3.1]{BGS94}.

Soient $A$ un anneau de Dedekind, $\sE$ un fibr\'e vectoriel de rang $n+1$ sur $\spec A$, et $\sE^\vee$ le fibr\'e dual de $\sE$. \`A tout sch\'ema $L$ sur $\spec A$, le foncteur grassmannien associe l'ensemble de modules quotients localement libres de $\sE^\vee\otimes_{A}\O_L$ de rang $d+1$ sur $\O_L$. On d\'esigne par $\hbox{Gr}\left(d+1,\sE^\vee\right)$ le sch\'ema qui repr\'esente ce foncteur grassmannien. En particulier, si $d=0$, on le d\'esigne par $\mathbb P(\sE)$ pour des raisons de simplicit\'e.

Dans la suite, on introduit le plongement de Pl\"ucker $\hbox{Gr}(d+1,\sE^\vee)\rightarrow\mathbb P\left(\bigwedge^{d+1}\sE^\vee\right)$, et le sous-sch\'ema d'incidence $\widetilde{\Gamma}$ de $\mathbb P\left(\sE\right)\times_{\spec A}\mathbb P\left(\bigwedge^{d+1}\sE^\vee\right)$ sur $\spec A$. Pour toute $A$-alg\`ebre $k$ qui est un corps, le plongement de Pl\"ucker envoie un point de $\hbox{Gr}\left(d+1,\sE_k^\vee\right)$ dans $\mathbb P\left(\bigwedge^{d+1}\sE_k^\vee\right)$ par \eqref{plucker embedding}. De plus, les points de $\widetilde{\Gamma}$ \`a valeur dans $k$ sont les couples \[\left(\xi,\alpha\right)\in \mathbb P\left(\sE_k\right)(k)\times\mathbb P\left(\bigwedge^{d+1}\sE_k^\vee\right)(k)\] satisfaisant $\theta\left(\xi\otimes\alpha\right)=0$, o\`u $\theta$ est d\'efini dans \eqref{plucker embedding}.

Soient \[\widetilde{p}': \mathbb P\left(\sE\right)\times_{\spec A}\mathbb P\left(\bigwedge\nolimits^{d+1}\sE^\vee\right)\rightarrow\mathbb P\left(\sE\right)\] et \[\widetilde{q}': \mathbb P\left(\sE\right)\times_{\spec A}\mathbb P\left(\bigwedge\nolimits^{d+1}\sE^\vee\right)\rightarrow\mathbb P\left(\bigwedge\nolimits^{d+1}\sE^\vee\right)\] les deux projections canoniques. et $\widetilde{v}:\widetilde{\Gamma}\rightarrow\mathbb P\left(\sE\right)\times_{\spec A}\mathbb P\left(\bigwedge^{d+1}\sE^\vee\right)$ le plongement canonique. On d\'efinit \[\widetilde p=\widetilde{p}'\circ\widetilde v:\widetilde{\Gamma}\rightarrow\mathbb P\left(\sE\right)\hbox{ et }\widetilde q=\widetilde{q}'\circ\widetilde v:\widetilde{\Gamma}\rightarrow\mathbb P\left(\bigwedge\nolimits^{d+1}\sE^\vee\right).\]
Alors on a le diagramm commutatif
\[\xymatrix{ & \widetilde{\Gamma}\ar@/^2pc/[ddr]^{\widetilde q}\ar@/_2pc/[ddl]_{\widetilde p}\ar[d]^{\widetilde{v}}&\\ & \mathbb P(\sE)\times_{\spec A}\mathbb P\left(\bigwedge\nolimits^{d+1}\sE^\vee\right)\ar[dr]^{\widetilde q'}\ar[dl]_{\widetilde p'}&\\ \mathbb P(\sE)\ar[r]& \spec A&\mathbb P\left(\bigwedge\nolimits^{d+1}\sE^\vee\right).\ar[l]}\]
On a la proposition suivante concernant le cycle $\widetilde{q}_*\left( \widetilde{p}^{*}\left[\mathscr X\right]\right)$, dont la d\'emonstration est m\^eme que celle de \cite[Lemma 4.3.1]{BGS94}.
\begin{prop}\label{flatness of cayley form}
   Soit $\mathscr X$ un sous-sch\'ema ferm\'e de dimension pure de $\mathbb P\left(\sE\right)$. Si $\mathscr X$ est plat sur $\spec A$ (\resp plat sur $\spec A$ et irr\'eductible), $\widetilde{q}_*\left(\widetilde{p}^{*}\left[\mathscr X\right]\right)\rightarrow\spec A$ l'est aussi.
\end{prop}
\begin{proof}
D'abord, on suppose que $\mathscr X$ est un sch\'ema int\`egre. Un sch\'ema int\`egre sur $\spec A$ est plat si et seulement si son point g\'en\'erique se trouve au-dessus du point g\'en\'erique $\spec K$ de $\spec A$. Comme le morphisme $\widetilde{p}:\widetilde{\Gamma}\rightarrow \mathbb P(\sE)$ est lisse \`a fibres g\'eom\'etriquement connexes, le cycle $\widetilde{p}^{*}(\mathscr X)$ est irr\'eductible aussi. Son point g\'en\'erique se trouve au-dessus de celui de $\mathscr X$, et se trouve au-dessus de $\spec K$.

On d\'esigne par $mW$ le cycle $\widetilde{q}_*\left(\widetilde{p}^{*}[\mathscr X]\right)$, o\`u $m$ est un entier positif. Dans ce cas-l\`a, $W$ est plat sur $\spec A$. L'entier $m$ sera z\'ero si $W$ est de codimension plus grande que ou \'egale \`a $1$ dans $\widetilde{p}^{*}(\mathscr X)$ (cf. \cite[\S 1.4]{Fulton}), sinon $m$ s'indentifie au degr\'e du corps des fonctions rationnelles de $\widetilde{p}^{*}(\mathscr X)$ sur celui de $W$. Comme l'image directe propre commute avec l'image inverse (cf. \cite[Proposition 1.7 et \S20.1]{Fulton}), on est capacit\'e de calculer $m$ en appliquant le changement de base $\spec K\rightarrow\spec A$. Alors on r\'eduit le probl\`eme au cas o\`u l'annear de base est un corps, qui est d\'emontr\'e dans la proposition \ref{cayleyform}, voir \cite[Proposition 3.4]{Chen1} aussi.

Si $\mathscr X$ n'est pas int\`egre, soit $\left[\mathscr X\right]=\sum\limits_{i\in I} m_i[\mathscr X_i]$ le cycle fondamental de $\mathscr X$, alors on a
\[\widetilde{q}_*\left(\widetilde{p}^*\left[\mathscr X\right]\right)=\sum_{i\in I} m_i\widetilde{q}_*\left(\widetilde{p}^*[\mathscr X_i]\right).\]
Donc on a le r\'esultat en appliquant l'argument ci-dessus composante par composante.
\end{proof}

\begin{defi}\label{cayley form over dedekind ring}
 Soit $\mathscr X$ un sous-sch\'ema ferm\'e de $\mathbb P(\sE)$ sur $\spec A$. On dit que le cycle $\widetilde{q}_*\left( \widetilde{p}^{*}\left[\mathscr X\right]\right)$ est le \textit{diviseur de Cayley} de $\mathscr X$ sur $\mathbb P\left(\bigwedge^{d+1}\sE^\vee\right)$.
\end{defi}

La proposition suivante est autour de la commutativit\'e de la construction du diviseur de Cayley et certains changements de base, qui est \'enonc\'e dans \cite[\S 4.3.2 (i)]{BGS94}. On va donner une d\'emonstration d\'etaill\'ee de l'assertion pour le cas d'anneau de Dedekind.

\begin{prop}\label{compatible with flat base change}
  Avec toutes les notations et la constructions dans cette partie. Soit $\mathscr X$ un sous-sch\'ema ferm\'e de dimension pure de $\mathbb P(\sE)$ sur $\spec A$, o\`u $A$ est un anneau de Dedekind. Soit $T$ un sch\'ema noeth\'erien r\'egulier qui satisfait l'une des deux condition suivantes:
  \begin{description}
    \item[(i)] le changement de base $T\rightarrow \spec A$ est plat;
    \item[(ii)] le sch\'ema $T$ est un point ferm\'e dans le changement de base $T\rightarrow \spec A$, et $\mathscr X\times_{\spec A}T\rightarrow T$ et $\mathscr X\rightarrow\spec A$ ont la m\^eme dimension relative.
  \end{description}
  Alors on a
  \[\left(\widetilde{q}_{T}\right)_*\left(\left(\widetilde{p}_T\right)^{*}\left[\mathscr X_T\right]\right)=\left(\widetilde{q}_*\left(\widetilde{p}^{*}\left[\mathscr X\right]\right)\right)_T.\]
  Autrement dit, la construction du diviseur de Cayley commute au changement de base satisfaisant les condtions ci-dessus.
\end{prop}
\begin{proof}
Par la construction dans l'\'enonc\'e, on a le diagram commutatif
\[\xymatrix{\relax \mathbb P(\sE)_T\ar[d]^-{r_1}&\widetilde{\Gamma}_T\ar[l]_-{\widetilde{p}_T}\ar[d]^-{r_2}\ar[r]^-{\widetilde{q}_T}&\mathbb P\left(\bigwedge^{d+1}\sE^\vee\right)_T\ar[d]^-{r_3}\\
\mathbb P(\sE)&\widetilde{\Gamma}\ar[l]_-{\widetilde{p}}\ar[r]^-{\widetilde{q}}&\mathbb P\left(\bigwedge^{d+1}\sE^\vee\right)}\]
induit par le changement de base $T\rightarrow \spec A$.
\begin{description}
  \item[(i)] Si le changement de base $T\rightarrow \spec A$ est plat, les morphismes $r_1$, $r_2$ et $r_3$ sont plats aussi. Par d\'efinition, on a $\widetilde{p}\circ r_2=r_1\circ \widetilde{p}_T$, d'o\`u l'on a \[r_2^*(\widetilde{p}^*[\mathscr X])=\widetilde{p}_T^*(r_1^*[\mathscr X])\] par \cite[Lemma 1.7.1]{Fulton} coupl\'e avec \cite[\S20.1]{Fulton}.

      On d\'esigne $[\mathscr Y]=\widetilde{p}^*[\mathscr X]$ pour simplifier. Car le morphisme $\widetilde{q}$ est propre, on a \[r_3^*(\widetilde{q}_*[\mathscr Y])=(\widetilde{q}_{T})_*(r_2^*[\mathscr Y])\] d'apr\`es \cite[Proposition 1.7]{Fulton} coupl\'e avec \cite[\S20.1]{Fulton}, qui termine la d\'emonstration.
  \item[(ii)] Maintenant on consid\`ere le cas satisfaisant la condition (ii). Dans ce cas-l\`a, $T$ est un diviseur de Cartier sur $\spec A$. Alors on peut identifier $\mathbb P(\mathcal E)_T$ (\resp $\widetilde{\Gamma}_T$ et $\mathbb P\left(\bigwedge^{d+1}\sE^\vee\right)_T$) \`a un sous-sch\'ema ferm\'e de codimension $1$ de $\mathbb P(\mathcal E)$ (\resp $\widetilde{\Gamma}$ et $\mathbb P\left(\bigwedge^{d+1}\sE^\vee\right)$), \`a savoir la fibre de $\mathbb P(\mathcal E)$ (\resp $\widetilde{\Gamma}$ et $\mathbb P\left(\bigwedge^{d+1}\sE^\vee\right)$) au-dessus de l'image de $T\rightarrow\spec A$. Alors les morphismes $r_1$, $r_2$ et $r_3$ sont des immersions ferm\'ees.

      Avant tout, on suppose que $\mathscr X$ est int\`egre. D'abord, on d\'emontrera $\mathscr X\nsubseteq\mathbb P(\sE)_T$ comme des sous-sch\'emas ferm\'e de $\mathbb P(\sE)$. Si $\mathscr X\subseteq\mathbb P(\sE)_T$, alors $\mathscr X_T$ a la m\^eme dimension que celle de $\mathscr X$. La dimension de $\mathbb P(\sE)_T$ et $1$ plus petite que celle de $\mathbb P(\sE)$, donc on a une contradiction \`a partir de la condition sur la dimension relative dans (ii).

      Alors on a $\mathscr X\nsubseteq\mathbb P(\sE)_T$. Par la construction dans \cite[Definition 2.3, Remark 2.3]{Fulton} coupl\'e avec \cite[\S20.1]{Fulton}, on a
      \[[\mathscr X]\cdot[\mathbb P(\sE)_T]=[\mathscr X\cap \mathbb P(\sE)_T]=[r_1^{-1}\left(\mathscr X\right)]=r_1^*[\mathscr X]\]
      comme des cycles sur $\mathbb P(\sE)_T$.

      De plus, les morphismes $\widetilde{p}$ et $\widetilde{p}_T$ sont lisses et ont les fibres g\'eom\'etriquement connexes, donc l'inverse sch\'ematique $\mathscr Y=\widetilde{p}^{-1}(\mathscr X)$ est irr\'eductible, d'o\`u l'on a $[\mathscr Y]=\widetilde{p}^*[\mathscr X]$. Alors par le m\^eme argument que celui ci-dessus, on a
      \[r_2^*[\mathscr Y]=[\mathscr Y]\cdot[\widetilde{\Gamma}_T].\]
      comme des cycles sur $\widetilde{\Gamma}_T$. De plus, par \cite[Definition 2.3, Remark 2.3]{Fulton} encore coupl\'e avec \cite[\S20.1]{Fulton}, on a
      \[[\mathscr Y]\cdot[\widetilde{\Gamma}_T]=[\widetilde{p}^{-1}(\mathscr X)\cap\widetilde{p}_T^{-1}(\mathbb P(\sE)_T)]=[\widetilde{p}^{-1}(\mathscr X\cap\mathbb P(\sE)_T)]=\widetilde{p}^*_T(r_1^*[\mathscr X])\]
      comme des cycles sur $\widetilde{\Gamma}_T$.

       Par les arugments ci-dessus, on obtient \[r_2^*(\widetilde{p}^*[\mathscr X])=\widetilde{p}_T^*(r_1^*[\mathscr X]),\]
       qui signifie
       \[(\widetilde{p}^*[\mathscr X])_T=\widetilde{p}_T^*([\mathscr X_T]).\]

      Dans la suite, on a $\mathscr Y\nsubseteq\widetilde{\Gamma}_T$, consid\'er\'es commes des sous-sch\'emas ferm\'es de $\widetilde{\Gamma}$ par la condition sur les dimensions relatives dans (ii) aussi. Car le morphism $\widetilde{q}$ est propre, on a
      \[\widetilde{q}_*\left([\mathscr Y]\cdot[\widetilde{\Gamma}_T]\right)=\widetilde{q}_*\left[\mathscr Y\right]\cdot\left[\mathbb P\left(\bigwedge\nolimits^{d+1}\sE^\vee\right)_T\right]\]
       par la formule de projection de cycles dans la th\'eorie de l'intersection (cf. \cite[Chap. 5, C), \S7, (10)]{LNM11}), d'o\`u l'on a \[(\widetilde{q}_{T})_*[\mathscr Y_T]=(\widetilde{q}_*[\mathscr Y])_T.\]

      Si $\mathscr X$ n'est pas int\`egre, on a le r\'esultat en appliquant l'argument ci-dessus composante par composante.
\end{description}

\end{proof}

\begin{rema}
  Dans \cite[\S 4.3.2 (i)]{BGS94}, on remplace $\spec A$ dans la proposition \ref{compatible with flat base change} par un sch\'ema noeth\'erien r\'egulier $T$, et on consid\`ere un plongement ferm\'e r\'egulier de $T'$ dans $T$. Dans ce cas-l\`a, la formation de la vari\'et\'e de Chow encore commute au changement de base $T'\rightarrow T$ si on suppose la m\^eme condition sur la dimension relative que celle dans la proposition \ref{compatible with flat base change} (ii).
\end{rema}
\section{Hauteurs d'un sch\'ema projectif}
Dans cette section, on introduira certaines fonctions hauteur et les comparera.
\subsection{Pr\'eliminaires}
Afin d'introduire des fonctions hauteur, d'abord on introduit certaines notions de base de la th\'eorie alg\'ebrique des nombres. On utilisera ces notions dans tout l'article sauf mention contraire.

Soient $K$ un corps de nombres, et $\O_K$ l'anneau des entiers de $K$. Dans tout l'article, on d\'esigne par $M_{K,f}$ l'ensemble des places finies de $K$, par $M_{K,\infty}$ l'ensemble des places infinies de $K$ et par $M_K$ l'ensemble des places de $K$.

\subsubsection*{Fibr\'e vectoriel norm\'e}
On appelle \textit{fibr\'e vectoriel norm\'e} sur $\spec\O_K$ toute donn\'ee $\overline {\sE}=\left(\sE,h\right)$, o\`u:
\begin{enumerate}
  \item $\sE$ est un $\O_K$-module projectif de rang fini;
  \item $h=(\|\ndot\|_v)_{v\in M_{K,\infty}}$ est une famille de normes, o\`u $\|\ndot\|_v$ est une norme sur $\sE\otimes_{\O_{K,v}}\mathbb C$ qui est invariante sous l'action du groupe $\gal (\mathbb C/K_v)$.
\end{enumerate}

Le rang de $\overline {\sE}$ est d\'efini comme celui de $\sE$. Si toutes les normes dans $h$ sont hermitiennes, on dit que $\overline{\sE}$ est un \textit{fibr\'e vectoriel hermitien} sur $\spec \O_K$. Si le rang de $\overline {\sE}$ est $1$, on dit que $\overline{\sE}$ est un fibr\'e en driotes hermitien sur $\spec\O_K$.

Soient $\overline L=(L,(\|\ndot\|_v)_{v\in M_{K,\infty}})$ un fibr\'e en droites hermitien sur $\spec\O_K$, et $s\in L\otimes_{\O_K}K$ un \'el\'ement non-nul. Pour une place $v\in M_{K,f}$, si $\Q_v$ est le corps $p$-adique, on d\'efinit $|a|_v=|N_{K_v/\Q_v}(a)|_p^{1/[K_v:\Q_v]}$, o\`u $|\ndot|_p$ est la valeur $p$-adique. De plus, on d\'efinit la norme $\|s\|_v=\inf\left\{|a|_v| a\in K_v^\times, a^{-1}s\in L\otimes_{\O_K}\widehat{\O}_{K,v}\right\}$ donn\'ee par le mod\`ele.

 \subsubsection*{Fonction hauteur}
 Soient $\overline{\sE}$ un fibr\'e vectoriel hermitien de rang $n+1$ sur $\spec\O_K$, et $\sE_K=\sE\otimes_{\O_K}K$. On d\'esigne par $\mathrm{Chow}_{d,\delta}^n(K)$ l'ensemble des sous-sch\'emas ferm\'es de $\mathbb P(\sE_K)$, qui sont de dimension pure $d$ et de degr\'e $\delta$ plong\'es dans $\mathbb P(\sE_K)$. Soit $\overline{\mathcal L}=\left(\mathcal L,(\|\ndot\|_v)_{v\in M_{K,\infty}}\right)$ un fibr\'e en droites ample arithm\'etique hermitien sur $\mathbb P(\sE)$. Alors la hauteur d'un sch\'ema projectif par rapport au fibr\'e en droites hermitien $\overline{\mathcal L}$ est une fonction
\[h_{\overline{\mathcal L}}:\mathrm{Chow}_{d,\delta}^n(K)\rightarrow\mathbb R,\]
qui mesure la complexit\'e arithm\'etique d'un $K$-sch\'ema projectif.

Plusieurs fonctions hauteur de sch\'emas arithm\'etiques seront utilis\'ees dans cet article. Soient $X$ un sous-sch\'ema ferm\'e de $\mathbb P(\sE_K)$ de dimension pure, et $\mathscr X$ l'adh\'erence sch\'ematique de $X$ dans $\mathbb P(\sE)$. D'abord on introduira une hauteur de $\mathscr X$ pour le cas g\'en\'eral. Au cas o\`u $X$ est une hypersurface, on a quelques propri\'et\'es sp\'eciales.
\subsection{Hauteurs d'un sch\'ema projectif de dimension pure}
D'abord, on d\'efinit une fonction hauteur introduite par Faltings dans \cite[Definition 2.5]{Faltings91} par la th\'eorie de l'intersection arithm\'etrique. La th\'eorie de l'intersection arithm\'etique est d\'evelopp\'ee par Gillet et Soul\'e dans \cite{Gillet_Soule-IHES90}, voir \cite{Soule92} pour une introduction syst\'ematique de cette th\'eorie.
\begin{defi}[Hauteur arakelovienne]\label{arakelov height of projective variety}
Soient $\overline{\sE}$ un fibr\'e vectoriel hermitien de rang $n+1$ sur $\spec\O_K$, et $\overline {\mathcal L}$ un fibr\'e en droites hermitien sur $\mathbb P(\sE)$. Soient $X$ un sous-sch\'ema ferm\'e de dimension pure $d$ de $\mathbb P(\sE_K)$, et $\mathscr X$ l'adh\'erence sch\'ematique de $X$ dans $\mathbb P(\sE)$. La \textit{hauteur arakelovienne} de $X$ est d\'efinie comme le nombre de l'intersection arithm\'etique
\begin{equation*}
  \frac{1}{[K:\Q]}\adeg\left(\widehat{c}_1(\overline{\mathcal{L}})^{d+1}\cdot[\mathscr X]\right),
\end{equation*}
o\`u $\widehat{c}_1(\overline{\mathcal{L}})$ est la premi\`ere classe de Chern arithm\'etique de $\overline{\mathcal L}$. Cette hauteur est not\'ee comme $h_{\overline{\mathcal{L}}}(X)$ ou $h_{\overline{\mathcal{L}}}(\mathscr X)$.
\end{defi}

\subsection{Hauteur d'une hypersurface projective}
Soient $\overline{\sE}$ un fibr\'e vectoriel hermitien de rang $n+1$ sur $\spec\O_K$, et $f(T_0,\ldots,T_n)$ un polyn\^ome homog\`ene \`a coefficients dans $K$ de degr\'e $\delta$, alors
\begin{equation*}
  X=\proj\left(K[T_0,\ldots,T_n]/\left(f(T_0,\ldots,T_n)\right)\right)
\end{equation*}
est un sous-sch\'ema ferm\'e de $\mathbb P(\sE_K)$ de dimension $n-1$. Il est en fait une hypersurface de $\mathbb P(\sE_K)$ de degr\'e $\delta$ (cf. \cite[Proposition 7.6, Chap. I]{GTM52}). Dans cette partie, on discutera des hauteurs d'une hypersurface dans $\mathbb P(\sE_K)$.

Pour tout $v\in M_{K,\infty}$, on d\'esigne par $|\ndot|_v$ la valeur absolue \`a la place $v$ qui satisfait $|a|_v=|N_{K_v/\Q_v}(a)|^{1/[K_v:\Q_v]}$, o\`u $|\ndot|$ est la valeur absolue usuelle sur $\mathbb R$ ou $\mathbb C$.
\begin{defi}[Hauteur classique]\label{classical height of hypersurface}
Soit
  \[f(T_0,\ldots,T_n)=\sum_{\begin{subarray}{x}(i_0,\ldots,i_n)\in\mathbb N^{n+1}\end{subarray}}a_{i_0,\ldots,i_n}T_0^{i_0}\cdots T_n^{i_n}\]
  un polyn\^ome non-nul \`a coefficients dans $K$.
  La \textit{hauteur classique} $h(f)$ du polyn\^ome homog\`ene $f$ est d\'efinie comme
\begin{equation*}
  h(f)=\sum_{v\in M_K}\frac{[K_v:\Q_v]}{[K:\Q]}\log\max\limits_{\begin{subarray}{x}(i_0,\ldots,i_n)\in\mathbb N^{n+1}\end{subarray}}\{|a_{i_0,\ldots, i_n}|_v\}\in\mathbb R_+.
\end{equation*}
De plus, si $f$ est homog\`ene et $X$ est l'hypersurface de $\mathbb P(\sE_K)$ d\'efinie par $f$, on d\'efinit $h(X)=h(f)$ comme la \textit{hauteur classique} de l'hypersurface $X$.
\end{defi}
La hauteur introduite dans la d\'efinition \ref{classical height of hypersurface} est invariante sous l'extension finie de corps de nombres.

Afin d'introduire une autre fonction hauteur, on introduit la mesure de Mahler.

\begin{defi}[Mesure de Mahler]\label{Mahler measure}
Soit $f(T_1,\ldots,T_n)\in \C[T_1,\ldots,T_n]$ un polyn\^ome. On d\'efinit la \textit{mesure de Mahler} du polyn\^ome $f(T_1,\ldots,T_n)$ comme
\begin{equation*}
  M(f)=\exp\left(\int_{[0,1]^{n}}\log|f(e^{2\pi it_1},\ldots,e^{2\pi it_n})| dt_1\cdots dt_n\right),
\end{equation*}
o\`u $|\ndot|$ est la valeur absolue usuelle sur $\mathbb C$.
\end{defi}
Pour un corps de nombres $K$, soient $f(T_1,\ldots,T_n)\in K[T_1,\ldots,T_n]$ et $v:K\hookrightarrow\C$ un plongement. On d\'efinit
\begin{equation}\label{Mahler mesure of v}
  M(v(f))=\exp\left(\int_{[0,1]^{n}}\log|v(f)(e^{2\pi it_1},\ldots,e^{2\pi it_n})| dt_1\cdots dt_n\right)
\end{equation}
comme la mesure de Mahler du polyn\^ome $f$ par rapport au plongement $v$.

On va introduire la fonction hauteur ci-dessous, qui est originaire de \cite[D\'efinition 1.10]{Philippon86}.
\begin{defi}[Hauteur de Philippon]\label{Philippon height}
  Soit $X$ une hypersurface de $\mathbb P(\sE_K)$ d\'efinie par le polyn\^ome homog\`ene
  \[f(T_0,\ldots,T_n)=\sum_{\begin{subarray}{x}(i_0,\ldots,i_n)\in\mathbb N^{n+1}\\ i_0+\cdots+i_n=\delta\end{subarray}}a_{i_0,\ldots,i_n}T_0^{i_0}\cdots T_n^{i_n},\]
   la \textit{hauteur de Philippon} de $X$ est d\'efinie comme
  \begin{equation*}
    h_{Ph}(X):=\sum_{v\in M_{K,f}}\frac{[K_v:\Q_v]}{[K:\Q]}\log\|f\|_v+\frac{1}{[K:\Q]}\sum_{v\in M_{K,\infty}}\log M(v(f)),
  \end{equation*}
  o\`u l'on d\'efinit
  \begin{equation}\label{norm by model of polynomial}
    \|f\|_{v}=\max_{\begin{subarray}{x}(i_0,\ldots,i_n)\in\mathbb N^{n+1}\\ i_0+\cdots+i_n=\delta\end{subarray}}\{|a_{i_0,\ldots,i_n}|_{v}\}
  \end{equation}
  pour tout $v\in M_{K,f}$, et $M(v(f))$ est la mesure de Mahler de $f$ par rapport \`a la place $v\in M_{K,\infty}$ d\'efinie par \eqref{Mahler mesure of v} dans la d\'efinition \ref{Mahler measure}.
\end{defi}

Soient $\overline{\sE}$ un fibr\'e vectoriel hermitien sur $\spec\O_K$, et $s\in H^0\left(\mathbb P(\sE_K), \O_{\mathbb P(\sE_K)}(\delta)\right)$ une section globale non-nulle. Pour toute place infinie $v\in M_{K,\infty}$ fix\'ee, on d\'esigne par $\|\ndot\|_{v,\mathrm{FS}}$ la m\'etrique de Fubini-Study sur $\mathbb P(\sE_{K,v})(\C)$ par rapport \`a la place infinie $v$. De plus, on d\'efinit
\begin{equation}\label{infinite norm}
\|s\|_{v,\infty}=\sup_{x\in \mathbb P(\sE_{K,v})(\C)}\|s(x)\|_{v,\mathrm{FS}}=\sup_{\|x\|_v=1}\|s(x)\|_{v,\mathrm{FS}}.
\end{equation}
 Soient $U\left(\sE_{K,v},\|\ndot\|_v\right)$ le groupe unitaire qui agit sur $\sE_{K,v}$, et $dv(x)$ une mesure $U\left(\sE_{K,v},\|\ndot\|_v\right)$-invariante unique probabiliste sur $\mathbb P(\sE_{K,v})(\C)$, ce qui signifie
 \[\int_{\mathbb P(\sE_{K,v})(\C)}dv(x)=1.\]

En suite, on d\'efinit
\begin{equation}\label{0-norm}
  \|s\|_{v,0}=\exp\left(\int_{\mathbb P(\sE_{K,v})(\C)}\log\|s(x)\|_{v,\mathrm{FS}}dv(x)\right).
\end{equation}
Pour tout nombre r\'eel strictement positif $p$, on d\'efinit
\begin{equation}\label{p-norm}
\|s\|_{v,p}=\left(\int_{\mathbb P(\sE_{K,v})(\C)}\|s(x)\|^p_{v,\mathrm{FS}}dv(x)\right)^{1/p}.
\end{equation}

Avec les norme \eqref{infinite norm}, \eqref{0-norm} et \eqref{p-norm} sur l'espace $H^0\left(\mathbb P(\sE_K), \O_{\mathbb P(\sE_K)}(\delta)\right)$, on d\'efinit la fonction hauteur suivante.
\begin{defi}[$p$-hauteur]\label{p-height}
  Soit $s\in H^0\left(\mathbb P(\sE_K),\O_{\mathbb P(\sE_K)}(\delta)\right)$ non-nulle. On d\'efinit la \textit{$p$-hauteur} de l'hypersurface $X$ de $\mathbb P(\sE_K)$ d\'efinie par la section globale $s$ comme
  \[h_p(X)=\sum_{v\in M_{K,f}}\frac{[K_v:\Q_v]}{[K:\Q]}\log\|s\|_v+\sum_{v\in M_{K,\infty}}\frac{[K_v:\Q_v]}{[K:\Q]}\log \|s\|_{v,p},\]
  o\`u $\|\ndot\|_v$ est la m\^eme que \eqref{norm by model of polynomial} pour $v\in M_{K,f}$ lorsque $s$ est consid\'er\'e comme un polyn\^ome homog\`ene, et la norme $\|\ndot\|_{v,p}$ est d\'efinie dans les \'egalit\'es \eqref{infinite norm}, \eqref{0-norm} et \eqref{p-norm} pour les $p\in[0,+\infty]$.
\end{defi}

\subsubsection*{Hauteur de la vari\'et\'e de Cayley}
On a d\'ej\`a d\'efini la vari\'et\'e de Cayley dans \S\ref{chow form and cayley form}. Dans cette partie, on va \'etudier la hauteur de vari\'et\'e de Cayley.

On consid\`ere le fibr\'e vectoriel hermitien
  \begin{equation}\label{sE with l^2}
    \overline{\sE}=\left(\O_K^{\oplus(n+1)},(\|\ndot\|_v)_{v\in M_{K,\infty}}\right)
  \end{equation}sur $\spec\O_K$, qui est muni des $\ell^2$-normes d\'efinies suivantes: pour tout plongement $v:K\hookrightarrow\C$, la norme $\|\ndot\|_v$ envoie le point $(x_0,\ldots,x_n)$ sur $\sqrt{|v(x_0)|^2+\cdots+|v(x_n)|^2}$.

Soit $X$ un sous-sch\'ema de dimension pure de dimension $d$ et degr\'e $\delta$ dans $\mathbb P(\sE_K)$ avec le fibr\'e hermitien $\overline{\sE}$ sur $\spec\O_K$ dans \eqref{sE with l^2}. On d\'esigne par $\widetilde h_{0}(X)$ la $0$-hauteur de la vari\'et\'e de Cayley de $X$ d\'efinie dans la d\'efinition \ref{p-height}. D'apr\`es \cite[Theorem 4.3.8]{BGS94}, on a
  \begin{equation}\label{0-height vs arakelov height}
    \widetilde h_{0}(X)= h_{\overline{\O(1)}}(X)-\frac{1}{2}\delta\mathcal H_N,
  \end{equation}
  o\`u $N=\rg\left(\bigwedge^{d+1}\sE_K\right)-1={n+1\choose d+1}-1$, $\mathcal H_N=1+\frac{1}{2}+\cdots+\frac{1}{N}$, et $\overline {\O(1)}$ est muni des m\'etriques de Fubini-Study \`a partir de $\overline{\sE}$ ci-dessus.
\subsubsection*{Sur la comparaison des hauteurs}
Dans cette partie, on comparera certaines hauteurs utiles de vari\'et\'es arithm\'etiques. Soient $\overline {\mathcal E}$ le fibr\'e hermitien d\'efini dans \eqref{sE with l^2}, et $s\in H^0\left(\mathbb P(\sE_K),\O_{\mathbb P(\sE_K)}(\delta)\right)$ une section globale non-nulle. On consid\`ere une telle section comme un polyn\^ome homog\`ene de degr\'e $\delta$ dans $K[T_0,\ldots,T_n]$. D'apr\`es \cite[Th\'eor\`eme 1]{Philippon91}, on a
\[0\leqslant \log M(v(s))-\log \|s\|_{v,0}\leqslant4\delta\log (n+1)\]
pour toute place $v\in M_{K,\infty}$, o\`u $M(v(s))$ est la mesure de Mahler de la section $s$ consid\'er\'ee comme un polyn\^ome homog\`ene de degr\'e $\delta$ par rapport \`a cette place, voir la d\'efinition \ref{Mahler measure} et \eqref{Mahler mesure of v} pour la d\'efinition.

Soit $X$ une hypersurface de $\mathbb P(\sE_K)$ de degr\'e $\delta$, alors on en d\'eduit
\begin{equation}\label{philippon height vs 0-height}
  0\leqslant h_{Ph}(X)-h_0(X)\leqslant 4\delta\log (n+1),
\end{equation}
voir la d\'efinition \ref{Philippon height} et la d\'efinition \ref{p-height} pour les d\'efinitions des deux hauteurs dans l'in\'egalit\'e \eqref{philippon height vs 0-height}.

Pour comparer la hauteur classique et la hauteur de Philippon d'une hypersurface, il faut comparer la mesure de Mahler et la valeur absolue maximale des coefficients du polyn\^ome qui d\'efinit l'hypersurface, o\`u une place infinie $v\in M_{K,\infty}$ est fix\'ee. On utilise la m\'ethode dans \cite[\S B.7]{Hindry}.
\begin{prop}\label{na\"ive height vs philippon height}
  Soient $\overline{\sE}$ d\'efini dans \eqref{sE with l^2}, et $X$ une hypersurface de $\mathbb P(\sE_K)$ de degr\'e $\delta$. Alors on a
   \[-\frac{1}{2}\log((n+1)(\delta+1))\leqslant h(X)-h_{Ph}(X)\leqslant(n+1)\delta\log2,\]
  o\`u la hauteur de Philippon $h_{Ph}(X)$ de $X$ est d\'efinie dans la d\'efinition \ref{Philippon height}, et la hauteur classique $h(X)$ de $X$ est d\'efinie dans la d\'efinition \ref{classical height of hypersurface}.
\end{prop}
\begin{proof}
On suppose que $X$ est d\'efini par le polyn\^ome
\[f(T_0,\ldots,T_n)=\sum_{\begin{subarray}{x}(i_0,\ldots,i_n)\in\mathbb N^{n+1}\\ i_0+\cdots+i_n=\delta\end{subarray}}a_{i_0,\ldots,i_n}T_0^{i_0}\cdots T_n^{i_n},\]
et $d_i=\deg_{T_i}(f)$ pour tout $i=0,\ldots,n$. D'apr\`es \cite[Lemma B.7.3.1, Lemma B.7.3.2]{Hindry}, soit $v\in M_{K,\infty}$, on a
\begin{eqnarray*}
  \prod_{v\in M_{K,\infty}}\frac{M(v(f))}{\sqrt{(d_0+1)\cdots(d_n+1)}}&\leqslant&\prod_{v\in M_{K,\infty}}\max_{\begin{subarray}{x}(i_0,\ldots,i_n)\in\mathbb N^{n+1}\\ i_0+\cdots+i_n=\delta\end{subarray}}\{|v(a_{i_0,\ldots,i_n})|\}\\&\leqslant&\prod_{v\in M_{K,\infty}}2^{(n+1)\delta}M(v(f)),
\end{eqnarray*}
o\`u la mesure de Mahler $M(v(f))$ est d\'efinie dans \eqref{Mahler mesure of v}. Si $v\in M_{K,f}$, alors
\[\max_{\begin{subarray}{x}(i_0,\ldots,i_n)\in\mathbb N^{n+1}\\ i_0+\cdots+i_n=\delta\end{subarray}}\{|a_{i_0,\ldots,i_n}|_v\}=\|f\|_v.\]
Donc par la d\'efinition \ref{Philippon height} et la d\'efinition \ref{classical height of hypersurface}, on obtient le r\'esultat, car $d_i\leqslant\delta$ pour tout $i\in\{0,1,\ldots,n\}$.
\end{proof}
On a le r\'esultat suivant en combinant les estimations ci-dessus.
\begin{prop}\label{comparing heights}
  Soient $\overline{\mathcal E}$ d\'efini dans \eqref{sE with l^2}, $X$ un sous-sch\'ema ferm\'e de dimension pure de $\mathbb P(\sE_K)$ de dimension $d$ et de degr\'e $\delta$, et $\psi_{X}\in\sym^\delta\left(\bigwedge^{d+1}\sE_K^\vee\right)$ l'\'el\'ement qui d\'efinit la vari\'et\'e de Cayley de $X$ (voir la proposition \ref{cayleyform} et la d\'efinition \ref{cayley form 2}). Alors on a
  \begin{eqnarray*}
  -\frac{1}{2}\log((N+1)(\delta+1))-\frac{1}{2}\delta\mathcal H_{N}&\leqslant& h(\psi_{X})-h_{\overline{\O(1)}}(X)\\
  &\leqslant&(N+1)\delta\log2+4\delta\log (N+1)-\frac{1}{2}\delta\mathcal H_{N},
  \end{eqnarray*}
  o\`u $h(\psi_{X})$ est d\'efini dans la d\'efinition \ref{classical height of hypersurface}, $h_{\overline{\O(1)}}(X)$ est d\'efini dans la d\'efinition \ref{arakelov height of projective variety}, $N={n+1\choose d+1}-1$, et $\mathcal H_N=1+\cdots+\frac{1}{N}$.
\end{prop}
  \begin{proof}
  Soit $X'$ l'hypersurface projective d\'efinie par $\psi_{X}$ dans $\mathbb P\left(\bigwedge^{d+1}\sE_K^\vee\right)$.
  On compare la hauteur de Philippon (voir la d\'efinition \ref{Philippon height}) de $X'$ et la hauteur classique de $X'$. D'apr\`es la proposition \ref{na\"ive height vs philippon height}, on a
  \begin{equation}\label{inequality naive-philippon}
   -\frac{1}{2}\log((N+1)(\delta+1))\leqslant h(\psi_{X})-h_{Ph}(X')\leqslant(N+1)\delta\log2.
  \end{equation}
  On compare la $0$-hauteur (voir la d\'efinition \ref{p-height}) de $X'$ et la hauteur de Philippon de $X'$. D'apr\`es (\ref{philippon height vs 0-height}), on a
  \begin{equation}\label{inequality philippon-0}
 0\leqslant h_{Ph}(X')-\widetilde{h}_0(X')\leqslant4\delta\log (N+1).
  \end{equation}
  Par (\ref{0-height vs arakelov height}), on a
  \begin{equation}\label{inequality 0-arakelov}
  \widetilde{h}_0(X')=h_{\overline{\O(1)}}(\mathscr X)-\frac{1}{2}\delta\mathcal H_{N}.
  \end{equation}
On combine \eqref{inequality naive-philippon}, \eqref{inequality philippon-0} et \eqref{inequality 0-arakelov}, on obtient le r\'esultat.
  \end{proof}

\subsection{Hauteur ad\'elique}\label{section of adelic height}
Dans cette partie, on introduira une fonction hauteur de la version ad\'elique d'une hypersurface.
\subsubsection*{Rappel de l'anneau ad\'elique}
D'abord, on rappelle la d\'efinition de l'anneau ad\'elique. \'Etant donn\'es un corps de nombre $K$ et son anneau des entiers $\O_K$, on d\'esgine par
\[\mathbb A_K=\left\{(a_v)_v\in\prod_{v\in M_K}K_v\mid\;a_v\in\O_{K,v}\hbox{ sauf pour un nombre fini de }v\in M_{K,f}\right\}\]
l'anneau ad\'elique de $K$,
 et par
 \[\mathbb{A}_{\O_K}=\left\{(a_v)_v\in \mathbb A_K|\;a_v\in\O_{K,v}\hbox{ pour tout }v\in M_{K,f}\right\}\]
 l'anneau ad\'elique des entiers de $K$, voir \cite[Chap. VI, \S1]{Neukirch} pour une introduction autonome de ces notions.

Soit $c\in K$. On d\'esigne par $\Delta(c)$ son image dans $\mathbb A_K$ par rapport au plongement diagonal $\Delta:K\hookrightarrow\mathbb A_K$. De plus, soit $a=(a_v)_{v\in M_K}\in\mathbb A_K$, on d\'efinit
\begin{equation}\label{adelic norm}
  |a|_{\mathbb A_K}=\prod_{v\in M_K}|a_v|_v^{[K_v:\Q_v]}.
\end{equation}
 Par la formule de produit (cf. \cite[Chap III, (1.3)Proposition]{Neukirch}), on a $|\Delta(a)|_{\mathbb A_K}=1$ pour tout $a\in K^\times$.

  On a le lemme suivant sur l'existance d'un \'el\'ement particulier dans $\mathbb A_K$.
 \begin{lemm}\label{denominator adelic}
   Soit $\mathfrak a$ un id\'eal fractionnaire de $\O_K$ avec la d\'ecomposition $\p_1^{n_1}\cdots\p_k^{n_k}$, o\`u $\p_1,\ldots,\p_k$ sont des id\'eaux primiers de $\O_K$ et $n_1,\ldots,n_k\in\mathbb Z$. Alors il existe un \'el\'ement $a=(a_v)_{v\in M_K}\in\mathbb A_K$, tel que $|a|_{\mathbb A_K}=1$ et \`a la place $v\in M_{K,f}$ par rapport \`a $\p_i$, $a_v$ engendre le m\^eme id\'eal fractionnaire que $\p_i^{n_i}\O_{K,\p_i}$ dans $K_{\p_i}$.
 \end{lemm}
\begin{proof}
  Pour chaque $\p_i\in\spm\O_K$, soit $\omega_i$ une uniformisante de $\O_{K,\p_i}$, dont l'existance est en raison du fait que $\O_{K,\p_i}$ est un anneau principal. Afin de construire un tel $a=(a_v)_{v\in M_K}$, pour une place $v\in M_{K,f}$ par rapport \`a $\p_i$, on pose $a_v=\omega_i^{n_i}$. Pour les autre places finies $v$, on pose $a_v=1$.

  Pour toute place $v\in M_{K,\infty}$, on a $K_v\cong \mathbb R$ ou $\mathbb C$. Alors on pose
  \[a_v=\left|\prod_{i=1}^k\#(\O_K/\p_i)^{n_i}\right|^{1/\left[K:\Q\right]}\]
pour toute place $v\in M_{K,\infty}$. Par un calcul \'el\'ementaire, cet \'el\'ement $(a_v)_{v\in M_K}$ satisfait l'assertion.
\end{proof}
\subsubsection*{Hauteur sur l'anneau ad\'elique}
Dans la suite, on consid\`ere un polyn\^ome sur l'anneau ad\'elique. Pour introduire une fonction hauteur, on va d\'efinir ses parties locales comme suit.
\begin{defi}[Partie locale]\label{p-part}
Soient $\{a_{i_0,\ldots,i_n}\}=\{(a^v_{i_0,\ldots,i_n})_{v\in M_K}\}$ une famille finie des \'el\'ements de $\mathbb A_K$ avec les indices $(i_0,\ldots,i_n)\in\mathbb N^{n+1}$, et
\[F(T_0,\ldots,T_n)=\sum_{\begin{subarray}{x}(i_0,\ldots,i_n)\in\mathbb N^{n+1}\end{subarray}}a_{i_0,\ldots,i_n}T_0^{i_0}\cdots T_n^{i_n}\in\mathbb A_K[T_0,\ldots,T_n]\]
non-nul. Pour toute place $v\in M_K$, on d\'esigne par
\[F^{(v)}(T_0,\ldots,T_n)=\sum_{\begin{subarray}{x}(i_0,\ldots,i_n)\in\mathbb N^{n+1}\end{subarray}}a^v_{i_0,\ldots,i_n}T_0^{i_0}\cdots T_n^{i_n}\]
 la \textit{$v$-partie} de $F(T_0,\ldots,T_n)$, ou par $F^{(\p)}(T_0,\ldots,T_n)$ pour le $\p\in\spm\O_K$ correspondant \`a une place $v\in M_{K,f}$ qui est appel\'e la \textit{$\p$-partie} de $F(T_0,\ldots,T_n)$.
\end{defi}
Soient $F$ et $F^{(v)}$ les m\^emes que dans la d\'efinition \ref{p-part}. Pour une place $v\in M_{K}$, on d\'esigne
\[\|F\|_{v}=\|F^{(v)}\|_v=\max_{\begin{subarray}{x}(i_0,\ldots,i_{n})\in\mathbb N^{n}\end{subarray}}\left\{\left|a^v_{i_0,\ldots,i_{n}}\right|_v\right\},\]
ou par $\|F\|_{\p}$ et $\|F^{(\p)}\|_{\p}$ pour le $\p\in\spm\O_K$ correspondant \`a une place finie. De plus, pour une place infinie $v\in M_{K,\infty}$, on d\'esigne
\[\|F\|_{2,v}=\|F^{(v)}\|_{2,v}=\left(\sum_{\begin{subarray}{x}(i_0,\ldots,i_{n})\in\mathbb N^{n}\end{subarray}}\left|a^v_{i_0,\ldots,i_{n}}\right|_v^2\right)^{\frac{1}{2}}.\]
 Avec les notations au-dessus, on introduit une fonction hauteur comme ci-dessous.
\begin{defi}[Hauteur ad\'elique]\label{adelic height}
  Soit $F$ le m\^eme que dans la d\'efinition \ref{p-part}. Les deux \textit{hauteurs ad\'eliques} de $F$ sont d\'efinies comme
  \[H_{\mathbb A_K}(F)=\prod_{v\in M_{K}}\|F\|_v^{[K_v:\Q_v]},\]
  et
  \[H_{\mathbb A_K,2}(F)=\prod_{v\in M_{K,f}}\|F\|_v^{[K_v:\Q_v]}\cdot\prod_{v\in M_{K,\infty}}\|F\|_{2,v}^{[K_v:\Q_v]}.\]
  De plus, on d\'efinit les hauteurs logarithmiquement ad\'eliques comme
  \[h(F)=\frac{1}{[K:\Q]}\log H_{\mathbb A_K}(F)\hbox{ et }h_2(F)=\frac{1}{[K:\Q]}\log H_{\mathbb A_K,2}(F)\]
  respectivement.
\end{defi}
En suite, on introduit la notion de la hauteur infinie.
\begin{defi}[Hauteur infinie ad\'elique]\label{adelic infinite height}
  Soit $F$ le m\^eme que celui dans la d\'efinition \ref{p-part}. Les deux \textit{hauteurs infinies ad\'eliques} de $F$ sont d\'efinies comme
  \[H_{\infty,\mathbb A_K}(F)=\prod_{v\in M_{K,\infty}}\|F\|_v^{[K_v:\Q_v]}\mbox{ et }H_{\infty,\mathbb A_K,2}(F)=\prod_{v\in M_{K,\infty}}\|F\|_{2,v}^{[K_v:\Q_v]}.\]
  De plus, on d\'efinit les hauteurs infinies logarithmiquement ad\'eliques comme
  \[h_\infty(F)=\frac{1}{[K:\Q]}\log H_{\infty,\mathbb A_K}(F)\hbox{ et }h_{\infty,2}(F)=\frac{1}{[K:\Q]}\log H_{\infty,\mathbb A_K,2}(F)\]
  respectivement.
\end{defi}
\begin{rema}\label{compare adelic and classic height}
  Par un calcul \'el\'ementaire, la hauteur ad\'elique est invariante sous la multiplication d'un \'el\'ement $a\in\mathbb A_K$ avec $|a|_{\mathbb A_K}=1$, et est bien s\^ur invariante sous la multiplication d'un \'el\'ement d'un \'el\'ement dans $K^\times$ consid\'er\'e comme un \'el\'ement dans $\mathbb A_K$ par rapport au plongement diagonal $K\hookrightarrow\mathbb A_K$.

  Soient $f\in K[T_0,\ldots,T_n]$, et $F$ l'image canonique de $f$ dans $\mathbb A_K[T_0,\ldots,T_n]$ par rapport au plongement diagonal $K\hookrightarrow\mathbb A_K$. Alors on a
 \[H_K(f)=H_{\mathbb A_K}(F) \hbox{ et }H_{K,2}(f)=H_{\mathbb A_K,2}(F),\]
 o\`u $H_K(f)$ est d\'efini dans la d\'efinition \ref{classical height of hypersurface}.
\end{rema}
Soit
  \[F(T_0,\ldots,T_n)=\sum_{\begin{subarray}{x}(i_0,\ldots,i_n)\in\mathbb N^{n+1}\\ i_0+\cdots+i_n=\delta\end{subarray}}a_{i_0,\ldots,i_n}T_0^{i_0}\cdots T_n^{i_n}\in \mathbb A_K[T_0,\ldots,T_n].\]
Avec toutes les notations dans la d\'efinition \ref{adelic height}, pour tout $v\in M_{K,\infty}$, on a
  \begin{eqnarray}\label{comapre coefficients of h_1 and h_2}
  \max_{\begin{subarray}{x}(i_0,\ldots,i_n)\in\mathbb N^{n+1}\\ i_0+\cdots+i_n=\delta\end{subarray}}\{|a^v_{i_0,\ldots,i_n}|_v\}
  &\leqslant&\left(\sum_{\begin{subarray}{x}(i_0,\ldots,i_n)\in\mathbb N^{n+1}\\ i_0+\cdots+i_n=\delta\end{subarray}}|a^v_{i_0,\ldots,i_n}|_v^{2}\right)^{\frac{1}{2}}\\
  &\leqslant&{n+\delta\choose n}^{\frac{1}{2}}\max_{\begin{subarray}{x}(i_0,\ldots,i_n)\in\mathbb N^{n+1}\\ i_0+\cdots+i_n=\delta\end{subarray}}\{|a^v_{i_0,\ldots,i_n}|_v\}\nonumber\end{eqnarray}
par la d\'efinition directement. Donc on a le r\'esultat suivant imm\'ediatement.
\begin{lemm}\label{compare h and h_2}
  Avec toutes les notations dans la d\'efinition \ref{adelic height} et la d\'efinition \ref{adelic infinite height}. Soit $F\in \mathbb A_K[T_0,\ldots,T_n]$. Alors on a
  \[0\leqslant h_2(F)-h(F)\leqslant \frac{1}{2}\log {n+\delta\choose n}\]
  et
  \[0\leqslant h_{\infty,2}(F)-h_\infty(F)\leqslant \frac{1}{2}\log {n+\delta\choose n}.\]
\end{lemm}

\section{Une estimation de l'annulation de r\'esultant par r\'eductions}
Dans cette section, pour un polyn\^ome sur un corps de nombres $K$, on donnera une majoration de l'annulation d'un de ses r\'esultants par r\'eductions. Cette majoration d\'epend la hauteur, le degr\'e et le nombre des variables de ce polyn\^ome.
\subsection{R\'esultant sur un anneau}\label{f of variable T_n}
Soient $A$ un anneau, et
   \[f(T_0,\ldots,T_n)=\sum_{\begin{subarray}{x}(i_0,\ldots,i_n)\in\mathbb N^{n+1}\\ i_0+\cdots+i_n=\delta\end{subarray}}a_{i_0,\ldots,i_n}T_0^{i_0}\cdots T_n^{i_n}\in A[T_0,\ldots,T_n]\]
   un polyn\^ome homog\`ene de degr\'e $\delta$. On d\'esigne par $d_i=\deg_{T_i}(f)$ pour $i=0,\ldots,n$, et on suppose $d_n>0$ sans perte de g\'en\'eralit\'e. On \'ecrit $f(T_0,\ldots,T_n)$ sous la forme de
   \begin{equation}\label{variable T_n 0}
   f(T_0,\ldots,T_n)=s_{d_n}(T_0,\ldots,T_{n-1})T_n^{d_n}+\cdots+s_0(T_0,\ldots,T_{n-1}),
   \end{equation}
   o\`u l'on a $s_{d_n}(T_0,\ldots,T_{n-1})\neq0$.

     Par \eqref{variable T_n 0}, on consid\`ere $f(T_0,\ldots,T_n)$ comme un polyn\^ome \`a coefficients dans l'anneau $A[T_0,\ldots,T_{n-1}]$ de degr\'e $d_n$, alors $\frac{\partial f}{\partial T_n}$ est de degr\'e $d_n-1$ \`a coefficients dans $A[T_0,\ldots,T_{n-1}]$ si $A$ est de caract\'erisque nulle. Par le sens ci-dessus, le \textit{r\'esultant} de $f(T_0,\ldots,T_n)$ et $\frac{\partial f}{\partial T_n}$ est
   \begin{eqnarray}\label{def res}
   & &\res_{d_n}\left(f,\frac{\partial f}{\partial T_n}\right)\\
   &=&s_{d_n}\det\left(
      \begin{array}{ccccccc}
        1 & s_{d_n-1} & \cdots & s_0 &   &   &   \\
          & s_{d_n} & s_{d_n-1} & \cdots & s_0 &   &   \\
          &   & \ddots & \ddots & \ddots & \ddots &  \\
          &   &   & s_{d_n} & s_{d_n-1} & \cdots & s_0 \\
        d_n & (d_n-1)s_{d_n-1} & \cdots & s_1 &   &   &   \\
          & d_n s_{d} & (d-1)s_{d-1} & \cdots & s_1 &   &   \\
          &   & \ddots & \ddots & \ddots & \ddots &  \\
          &   &   & d_n s_{d_n} & (d_n-1)s_{d_n-1} & \cdots & s_1 \\
      \end{array}
    \right).\nonumber\end{eqnarray}

    On d\'esigne
    \begin{eqnarray}\label{def of res'}
      & &\res'_{d_n}\left(f,\frac{\partial f}{\partial T_n}\right)\\
      &=&\det\left(
      \begin{array}{ccccccc}
        1 & s_{d_n-1} & \cdots & s_0 &   &   &   \\
          & s_{d_n} & s_{d_n-1} & \cdots & s_0 &   &   \\
          &   & \ddots & \ddots & \ddots & \ddots &  \\
          &   &   & s_{d_n} & s_{d_n-1} & \cdots & s_0 \\
        d_n & (d_n-1)s_{d_n-1} & \cdots & s_1 &   &   &   \\
          & d_n s_{d_n} & (d_n-1)s_{d_n-1} & \cdots & s_1 &   &   \\
          &   & \ddots & \ddots & \ddots & \ddots &  \\
          &   &   & d_n s_{d_n} & (d_n-1)s_{d_n-1} & \cdots & s_1 \\
      \end{array}
    \right)\nonumber
    \end{eqnarray}
    pour simplifier. Par la construction ci-dessus, on a $\res'_{d_n}\left(f,\frac{\partial f}{\partial T_n}\right)\in A[T_0,\ldots,T_{n-1}]$.
\subsection{Polyn\^ome ad\'eliquement primitif}\label{adelicly primitive}
   Maintenant on travaille sur un corp de nombre $K$ et son anneau des entiers $\O_K$. Soit $f\in K[T_0,\ldots,T_n]$ un polyn\^ome homog\`ene non-nul de degr\'e $\delta$. On maintient toutes les notations m\^eme que celles dans \S \ref{f of variable T_n} lorsque $A$ est $K$, $\mathbb A_K$ ou $\mathbb A_{\O_K}$.

   On consid\`ere $f(T_0,\ldots,T_n)$ comme un polyn\^ome \`a coefficients dans $\mathbb A_K$ par rapport au plongement diagonal $\Delta:K\hookrightarrow \mathbb A_K$. Par le lemme \ref{denominator adelic}, il existe un \'el\'ement $c=(c_v)_{v\in M_K}\in \mathbb A_K$ avec $|c|_{\mathbb A_K}=1$ (voir \eqref{adelic norm} pour la d\'efinition de $|\ndot|_{\mathbb A_K}$), tel que pour toute place $v\in M_{K,f}$, on ait
   \[\max_{\begin{subarray}{x}(i_0,\ldots,i_n)\in\mathbb N^{n+1}\\ i_0+\cdots+i_n=\delta\end{subarray}}\left\{|c\Delta(a_{i_0,\ldots,i_n})|_v\right\}=1.\]
   On d\'esigne par $b_{i_0,\ldots,i_n}=(b_{i_0,\ldots,i_n}^v)_v=c\Delta(a_{i_0,\ldots,i_n})$ pour simplifier, et on dit que
     \begin{equation}\label{variable T_n adelic}F(T_0,\ldots,T_n)=\sum_{\begin{subarray}{x}(i_0,\ldots,i_n)\in\mathbb N^{n+1}\\ i_0+\cdots+i_n=\delta\end{subarray}}b_{i_0,\ldots,i_n}T_0^{i_0}\cdots T_n^{i_n}\end{equation}
est le \textit{polyn\^ome ad\'eliquement primitif} de $f(T_0,\ldots,T_n)$. En effet, on a $F(T_0,\ldots,T_n)\in \mathbb A_{\O_K}[T_0,\ldots,T_n]$, et sa composante en toute place finie est un polyn\^ome primitif au sens habituel. Dans ce cas-l\`a, $F$ est de degr\'e $d_n$ en la variable $T_n$.

Pour un polyn\^ome ad\'eliquement primitif $F$, on a
\begin{equation}\label{h_infty=h}
  h(F)=h_\infty(F),\hbox{ et }h_2(F)=h_{\infty,2}(F)
\end{equation}
o\`u $h(F)$ et $h_2(F)$ sont d\'efinis dans la d\'efinition \ref{adelic height}, et $h_\infty(F)$ et $h_{\infty,2}(f)$ sont d\'efinis dans la d\'efinition \ref{adelic infinite height}.
\subsection{L'estimation d'une hauteur du r\'esultant}
    Soit
    \[G(T_0,\ldots,T_{n-1})=\sum_{\begin{subarray}{x}(i_0,\ldots,i_{n-1})\in\mathbb N^{n}\end{subarray}}b_{i_0,\ldots,i_{n-1}}T_0^{i_0}\cdots T_{n-1}^{i_{n-1}}\]
    un polyn\^ome non-nul \`a coefficients dans $\mathbb A_K$, o\`u l'on d\'esigne $b_{i_0,\ldots,i_{n-1}}=(b^v_{i_0,\ldots,i_{n-1}})_{v\in M_K}$. Pour tout $\p\in\spm\O_K$ correspondant \`a la place $v\in M_{K,f}$, on d\'efinit sa norme donn\'ee par le mod\`ele comme
    \[\left\|G\right\|_{\p}=\|G^{(\p)}\|_\p=\max_{\begin{subarray}{x}(i_0,\ldots,i_{n-1})\in\mathbb N^{n}\end{subarray}}\left\{|b_{i_0,\ldots,i_{n-1}}^v|_v\right\},\]
    o\`u la $\p$-partie $G^{(\p)}$ est d\'efinie dans la d\'efinition \ref{p-part}. Si $G\in \mathbb A_{\O_K}[T_0,\ldots,T_{n-1}]$, on d\'esigne
    \begin{equation}\label{R(F)}
      \mathcal P(G)=\{\p\in\spm\O_K|\;G^{(\p)}\mod \p[T_0,\ldots,T_{n-1}]=0\}.
    \end{equation}

Pour une place infinie $v\in M_{K,\infty}$, on d\'esigne
\[\|G\|_{v}=\max_{\begin{subarray}{x}(i_0,\ldots,i_{n-1})\in\mathbb N^{n}\end{subarray}}\left\{\left|b^v_{i_0,i_1,\ldots,i_{n-1}}\right|_v\right\}\hbox{ et }\|G\|_{2,v}=\left(\sum_{\begin{subarray}{x}(i_0,\ldots,i_{n-1})\in\mathbb N^{n}\end{subarray}}\left|b^v_{i_0,i_1,\ldots,i_{n-1}}\right|_v^2\right)^{\frac{1}{2}},\]
qui sont les m\^emes que ceux dans \S \ref{section of adelic height}.

Soient $G_1,\ldots,G_m\in\mathbb A_K[T_0,\ldots,T_{n-1}]$. Pour toute place $v\in M_{K,\infty}$, on d\'efinit
\begin{equation}\label{norm of vector of polynomial}
  \|(G_1,\ldots,G_m)\|_{2,v}=\sqrt{\|G_1\|^2_{2,v}+\cdots+\|G_m\|^2_{2,v}}.
\end{equation}

Pour le polyn\^ome $F(T_0,\ldots,T_n)$ d\'efini dans \eqref{variable T_n adelic}, on d\'efinit $\res'_{d_n}\left(F,\frac{\partial F}{\partial T_n}\right)\in\mathbb A_{\O_K}[T_0,\ldots,T_{n-1}]$ comme \eqref{def of res'}. Afin d'estimer la taille de l'ensemble $\mathcal P\left(\res'_{d_n}\left(F,\frac{\partial F}{\partial T_n}\right)\right)$ d\'efini dans \eqref{R(F)}, on a le r\'esultat suivant.
\begin{prop}\label{r\'esultant2}
  Avec toutes les notations et constructions au-dessus. Soit $F(T_0,\ldots,T_n)$ d\'etermin\'e dans \eqref{variable T_n adelic} sur $\mathbb A_{\O_K}$ \`a partir de $f(T_0,\ldots,T_n)$ dans \eqref{variable T_n 0}. Pour tout $\p\in\spm\O_K$, on d\'esigne $N(\p)=\#(\O_K/\p)$. Alors on a
  \begin{eqnarray*}
    & &\frac{1}{[K:\Q]}\sum_{\p\in\mathcal P\left(\res'_{d_n}\left(F,\frac{\partial F}{\partial T_n}\right)\right)}\log N(\p)\\
    &\leqslant&(2d_n-2)h(f)+(d_n-1)\log{n+\delta\choose n}+\log(2d_n^{d_n}-d_n^{d_n-1}).
  \end{eqnarray*}
  o\`u $h(f)$ est d\'efini dans la d\'efinition \ref{classical height of hypersurface}.
\end{prop}
\begin{proof}
Si $d_n=\max\{d_0,\ldots,d_n\}=1$, on a $\res'_{d_n}\left(F,\frac{\partial F}{\partial T_n}\right)=1$ par d\'efinition directement. Donc l'ensemble $\mathcal P\left(\res'_{d_n}\left(F,\frac{\partial F}{\partial T_n}\right)\right)$ est vide, qui satisfait l'in\'egalit\'e dans l'\'enonc\'e.

 Dans le reste de la d\'emonstration, on suppose $d_n\geqslant2$. Comme $\res'_{d_n}\left(F,\frac{\partial F}{\partial T_n}\right)\in\mathbb A_{\O_K}[T_0,\ldots,T_{n-1}]$, alors pour tout $\p\in\spm\O_K$, on a
  \[\left\|\res'_{d_n}\left(F,\frac{\partial F}{\partial T_n}\right)^{(\p)}\right\|_\p\leqslant1.\]
  Donc on a
  \begin{eqnarray*}
    & &\frac{1}{[K:\Q]}\sum_{\p\in\mathcal P\left(\res'_{d_n}\left(F,\frac{\partial F}{\partial T_n}\right)\right)}\log N(\p)\\
    &\leqslant&-\sum_{\p\in\mathcal P\left(\res'_{d_n}\left(F,\frac{\partial F}{\partial T_n}\right)\right)}\frac{[K_\p:\Q_\p]}{[K:\Q]}\log \left\|\res'_{d_n}\left(F,\frac{\partial F}{\partial T_n}\right)^{(\p)}\right\|_\p.
  \end{eqnarray*}
  De plus, l'in\'egalit\'e
  \[\left\|\res'_{d_n}\left(F,\frac{\partial F}{\partial T_n}\right)^{(\p)}\right\|_\p<1\]
  est v\'erifi\'ee si et seulement si $\p\in \mathcal P\left(\res'_{d_n}\left(F,\frac{\partial F}{\partial T_n}\right)\right)$. Alors on obtient
  \begin{eqnarray*}
    & &-\sum_{\p\in\mathcal P\left(\res'_{d_n}\left(F,\frac{\partial F}{\partial T_n}\right)\right)}\frac{[K_\p:\Q_\p]}{[K:\Q]}\log \left\|\res'_{d_n}\left(F,\frac{\partial F}{\partial T_n}\right)^{(\p)}\right\|_\p\\
    &=&-\sum_{\p\in\spm\O_K}\frac{[K_\p:\Q_\p]}{[K:\Q]}\log \left\|\res'_{d_n}\left(F,\frac{\partial F}{\partial T_n}\right)^{(\p)}\right\|_\p\\
    &=&-h\left(\res'_{d_n}\left(F,\frac{\partial F}{\partial T_n}\right)\right)+\sum_{v\in M_{K,\infty}}\frac{[K_v:\Q_v]}{[K:\Q]}\log \left\|\res'_{d_n}\left(F,\frac{\partial F}{\partial T_n}\right)^{(v)}\right\|_v.
  \end{eqnarray*}
En effet, le polyn\^ome $\res'_{d_n}\left(F,\frac{\partial F}{\partial T_n}\right)\in\mathbb A_K[T_0,\ldots,T_{n-1}]$ est obtenu par la multiplication d'un \'el\'ement $c\in\mathbb A_K$ \`a un polyn\^ome dans $K[T_0,\ldots,T_{n-1}]$ satisfaisant $|c|_{\mathbb A_K}=1$. Alors d'apr\`es la remarque \ref{compare adelic and classic height} et la d\'efinition \ref{classical height of hypersurface}, on a $h\left(\res'_{d_n}\left(F,\frac{\partial F}{\partial T_n}\right)\right)\geqslant0$. On le combine avec \eqref{comapre coefficients of h_1 and h_2}, et on a
  \begin{eqnarray*}& &-\sum_{\p\in\mathcal P\left(\res'_{d_n}\left(F,\frac{\partial F}{\partial T_n}\right)\right)}\frac{[K_\p:\Q_\p]}{[K:\Q]}\log \left\|\res'_{d_n}\left(F,\frac{\partial F}{\partial T_n}\right)^{(\p)}\right\|_\p\\
  &\leqslant&\sum_{v\in M_{K,\infty}}\frac{[K_v:\Q_v]}{[K:\Q]}\log \left\|\res'_{d_n}\left(F,\frac{\partial F}{\partial T_n}\right)^{(v)}\right\|_{2,v}.\end{eqnarray*}
  d'apr\`es l'\'egalit\'e ci-dessus.

  Afin d'estimer $\left\|\res'_{d_n}\left(F,\frac{\partial F}{\partial T_n}\right)^{(v)}\right\|_{2,v}$ pour une place $v\in M_{K,\infty}$, on renvoie que c'est la $v$-partie de $\res'_{d_n}\left(F,\frac{\partial F}{\partial T_n}\right)$ d\'efinie dans la matrice de \eqref{def of res'} sur l'anneau $\mathbb A_{\O_K}$. Soit $w^v_i$ le $i$-i\`eme vecteur de ligne dans la $v$-partie de la matrice supprimant la premi\`ere colonne dans \eqref{def of res'}, o\`u $i=1,\ldots,2d_n-1$. Alors on \'ecrit la matrice sous la forme de \begin{equation*}
  \res'_{d_n}\left(F,\frac{\partial F}{\partial T_n}\right)^{(v)}=\det\left(
                                                                      \begin{array}{cc}
                                                                        1 & w^v_1 \\
                                                                          & \vdots \\
                                                                          & w^v_{d_n-1} \\
                                                                        d_n & w^v_{d_n} \\
                                                                          & \vdots \\
                                                                          & w^v_{2d_n-1} \\
                                                                      \end{array}
                                                                    \right).
  \end{equation*}

  D'abord on d\'eveloppe la matrice ci-dessus par rapport \`a la permi\`ere colonne, qui a deux \'el\'ements non-nuls seulement. Soient $M_v$ et $N_v$ les cofacteurs de $1$ et $d_n$ respectivement dans cette matrice. Alors on a
  \[M_v=\left(
          \begin{array}{c}
            w^v_2 \\
            \vdots \\
            w^v_{2d_n-1} \\
          \end{array}
        \right)
   \hbox{ et }N_v=\left(
          \begin{array}{c}
            w^v_1 \\
            \vdots \\
            w^v_{d_n-1} \\
            w^v_{d_n+1}\\
            \vdots\\
            w^v_{2d_n-1}\\
          \end{array}
        \right).\]
  Par d\'efinition, on a
  \[\left\|\res'_{d_n}\left(F,\frac{\partial F}{\partial T_n}\right)^{(v)}\right\|_{2,v}\leqslant \|\det(M_v)\|_{2,v}+d_n\|\det(N_v)\|_{2,v}.\]

   En raison du fait que la norme $\|\ndot\|_{2,v}$ est hermitienne sur $K_v[T_0,\ldots,T_n]$, d'apr\`es \cite[Corollaire 1, \S3.5 Chap. V]{Bourbaki81}, on a
  \[\|\det(M_v)\|_{2,v}\leqslant\prod_{i=2}^{2d_n-1}\|w^v_i\|_{2,v}\hbox{ et }\|\det(N_v)\|_{2,v}\leqslant \frac{1}{\|w^v_{d_n}\|_{2,v}}\prod_{i=1}^{2d_n-1}\|w^v_i\|_{2,v},\]
  o\`u chaque $\|w_i^v\|_{2,v}$ est d\'efini dans \eqref{norm of vector of polynomial} pour $i=1,\ldots,2d_n-1$.

  Par d\'efinition, on a $\|w^v_i\|_{2,v}\leqslant \|F\|_{2,v}$ pour les $i=1,\ldots,d_n-1$, $\|w^v_{d_n}\|_{2,v}\leqslant(d_n-1)\|F\|_{2,v}$, et $\|w^v_i\|_{2,v}\leqslant d_n\|F\|_{2,v}$ pour les $i=d_n+1,\ldots,2d_n-1$. Alors on obtient
  \[\|\det(M_v)\|_{2,v}\leqslant(d_n-1)d_n^{d_n-1}\|F\|_{2,v}^{2d_n-2}\hbox{ et }d_n\|\det(N_v)\|_{2,v}\leqslant d_n^{d_n}\|F\|_{2,v}^{2d_n-2}.\]

  On combine les estimations ci-dessus, et on obtient
  \begin{eqnarray*}
    & &\sum_{v\in M_{K,\infty}}\frac{[K_v:\Q_v]}{[K:\Q]}\log \left\|\res'_{d_n}\left(F,\frac{\partial F}{\partial T_n}\right)^{(v)}\right\|_{2,v}\\
    &\leqslant&(2d_n-2)\sum_{v\in M_{K,\infty}}\frac{[K_v:\Q_v]}{[K:\Q]}\log\|F\|_{2,v}+\log(2d_n^{d_n}-d_n^{d_n-1}).
    \end{eqnarray*}
    Comme $F$ est le polyn\^ome ad\'eliquement primitif de $f$, on a
     \[\sum_{v\in M_{K,\infty}}\frac{[K_v:\Q_v]}{[K:\Q]}\log\|F\|_{2,v}=h_2(F)\]
    par \eqref{h_infty=h}. D'apr\`es la remarque \ref{compare adelic and classic height} et la proposition \ref{compare h and h_2}, on a
    \[h_2(F)\leqslant h(F)+\frac{1}{2}\log{n+\delta\choose n}=h(f)+\frac{1}{2}\log{n+\delta\choose n},\]
    o\`u et $h(f)$ est d\'efini dans la d\'efinition \ref{classical height of hypersurface}. Donc on obtient le r\'esultat.
    %&=&(2d_n-2)h_2(F)+\log(2d_n^{d_n}-d_n^{d_n-1})\\
    %&\leqslant&(2d_n-2)h(f)+(d_n-1)\log{n+\delta\choose n}+\log(2d_n^{d_n}-d_n^{d_n-1}),
  %\end{eqnarray*}
\end{proof}

\section{Contr\^ole des fibres non r\'eduites d'une hypersurface projective}
Soient $\overline{\sE}$ un fibr\'e vectoriel hermitien de rang $n+1$ sur $\spec\O_K$, $X$ un sous-sch\'ema ferm\'e de $\mathbb P(\sE_K)$, et $\mathscr X$ l'adh\'erence sch\'ematique de $X$ dans $\mathbb P(\sE)$. Par \cite[Th\'eor\`eme (9.7.7)]{EGAIV_3}, si $X$ est r\'eduit, alors il n'a y qu'un nombre fini d'ideaux maximaux $\p\in\spm\O_K$ telles que le fibre $\mathscr X_{\f_\p}=\mathscr X\times_{\spec\O_K}\spec \f_\p$ ne soit pas r\'eduite, o\`u $\f_\p$ est le corps r\'esiduel de $\O_K$ en $\p$.

Dans cette section, on donnera une description num\'erique des r\'eduitions non r\'eduites lorsque $X$ est une hypersurface dans $\mathbb P(\sE_K)$. Plus pr\'ecisement, on donnera une majoration du produit des normes des id\'eaux maximaux $\p\in\spm\O_K$ tels que $\mathscr X_{\f_\p}=\mathscr X\times_{\spec\O_K}\spec\f_\p$ ne soit pas r\'eduit.
\subsection{R\'esultats pr\'eliminaires}
 Pour le crit\`ere du r\'eduisant d'hypersurface et le choix du plongement dans $\mathbb P(\sE_K)$, il faut des r\'esultats auxiliaires suivants.

\subsubsection*{Crit\`ere du r\'eduisant d'une hypersurface}
Pour introduire une m\'ethode de crit\`ere de r\'eduisant d'une hypersurface projective, d'abord on r\'ef\'erence le r\'esultat suivant, qui est une crit\`ere du r\'eduisant de l'hypersurface affine.
\begin{lemm}[\cite{LiuQing}, Exercise 2.4.1]\label{affine hypersurface}
  Soient $k$ un corps, et $P\in k[T_1,\ldots,T_n]$ un polyn\^ome non-nul. Alors le sch\'ema $\spec\left(k[T_1,\ldots,T_n]/(P)\right)$ est r\'eduit (\resp irr\'eductible; \resp int\`egre) si et seulement si $P$ n'a pas de facteur carr\'e (\resp est une puissance d'un polyn\^ome irr\'eductible, \resp est irr\'eductible).
\end{lemm}
\begin{rema}\label{projective hypersurface lost its reduceness}
  D'apr\`es le lemme \ref{affine hypersurface}, soient $k$ un corps, et $P\in k[T_0,\ldots,T_n]$ un polyn\^ome homog\`ene non-nul. Alors $X=\proj\left(k[T_0,\ldots,T_n]/(P)\right)$ est r\'eduit si et seulement si $P$ n'a pas de facteur carr\'e.
\end{rema}
\subsubsection*{Changement de coordonn\'ee}
Soit $\mathscr X$ un $\O_K$-sch\'ema r\'eduit, on d\'efinit l'ensemble
\begin{equation}\label{Q(W)}
\mathcal Q(\mathscr X)=\{\p\in\spm\O_K|\;\mathscr X\times_{\spec \O_K}\spec \f_\p\hbox{ ne soit pas r\'eduite}\}.
\end{equation}
Afin d'appliquer la m\'ethode de r\'esultant pour l'estimation des fibres non r\'eduites d'un sch\'ema arithm\'etique, il faut choisir une coordonn\'ee particuli\`ere. Les deux lemmes suivants sont utiles pour ce but.

Le lemme premier est d\'eduit d'apr\`es \cite[Proposition(4.4.5), Chap. I]{NouveauEGA1} directement.
\begin{lemm}\label{non-reduced invariant}
  Soient $\sE$ un fibr\'e vectoriel hermitien sur $\spec\O_K$, $\mathscr X$ un sous-sch\'ema ferm\'e r\'eduit de $\mathbb P(\sE)$, et $\sigma\in\Aut_{\O_K}(\mathbb P(\sE))$. Alors on a $\mathcal Q(\mathscr X)=\mathcal Q(\sigma(\mathscr X))$, o\`u $\mathcal Q(\ndot)$ est d\'efini dans \eqref{Q(W)}.
\end{lemm}

Le lemme suivant est une analogie de \cite[Lemma 2.1, Lemma 2.2]{Fukshansky2006}.
\begin{lemm}\label{change of coordinate}
  Soit $F(T_0,\ldots,T_n)$ le polyn\^ome homog\`ene ad\'eliquement primitif determin\'e \`a \eqref{variable T_n adelic}. Alors il existe un $(a_0,\ldots,a_{n-1})\in\mathbb Z^n$ avec $\max\limits_{0\leqslant i\leqslant n-1}\{|a_i|\}\leqslant\frac{\delta+1}{2}$ qui induit un $\sigma\in\Aut_{\O_K}(\mathbb P(\sE))$, tel que le coefficient de $T_n^{\delta}$ dans $F(T_0+a_0T_n,\ldots,T_{n-1}+a_{n-1}T_{n-1},T_n)$ ne soit pas z\'ero, o\`u l'on prolonge $\mathbb Z$ dans $\mathbb A_{\O_K}$ canoniquement. De plus, on a
  \[h(\sigma(\mathscr X))\leqslant h(X)+\log\left({n+\delta\choose\delta}{\delta\choose\left[\frac{\delta+1}{2}\right]}\left(\frac{\delta+1}{2}\right)^\delta\right).\]
\end{lemm}
\begin{proof}
  En effet, le coefficient de $T_0^{s_0}\cdots T_n^{s_n}$ dans le polyn\^ome
  \begin{eqnarray*}
    & &F(T_0+a_0T_n,\ldots,T_{n-1}+a_{n-1}T_{n-1},T_n)\\
    &=&\sum_{\begin{subarray}{x}(j_0,\ldots,j_n)\in\mathbb N^{n+1}\\ j_0+\cdots+j_n=\delta\end{subarray}}b_{j_0,\ldots,j_n}(T_0+a_0T_n)^{j_0}\cdots(T_{n-1}+a_{n-1}T_n)^{j_{n-1}}T_n^{j_n}
  \end{eqnarray*}
   est
  \begin{equation}\label{coefficient after automorphism}
    \sum_{\begin{subarray}{x}(j_0,\ldots,j_n)\in\mathbb N^{n+1}\\ j_0+\cdots+j_n=\delta\end{subarray}}b_{j_0,\ldots,j_n}{j_0\choose s_0}\cdots{j_{n-1}\choose s_{n-1}}a_0^{j_0-s_0}\cdots a_{n-1}^{j_{n-1}-s_{n-1}}
  \end{equation}
  par un calcul \'el\'ementaire. En particulier, le coefficient de $T_n^\delta$ dans $F(T_0+a_0T_n,\ldots,T_{n-1}+a_{n-1}T_{n-1},T_n)$ est
  $$\sum_{\begin{subarray}{x}(i_0,\ldots,i_n)\in\mathbb N^{n+1}\\ i_0+\cdots+i_n=\delta\end{subarray}}b_{i_0,\ldots,i_n}a_0^{i_0}\cdots a_{n-1}^{i_{n-1}},$$
  qui est un polyn\^ome non-nul de degr\'e au plus $\delta$ de variables $a_0,\ldots,a_{n-1}$. De plus, l'ensemble $\left\{-\left[\frac{\delta+1}{2}\right],\ldots,-1,0,1,\ldots,\left[\frac{\delta+1}{2}\right]\right\}^n$ est de cardinal plus grand que ou \'egal \`a $(\delta+1)^n$. Par \cite[Lemma 1, la page 261]{Cassels1959}, il existe un $(a_0,\ldots,a_{n-1})\in \left\{-\left[\frac{\delta+1}{2}\right],\ldots,-1,0,1,\ldots,\left[\frac{\delta+1}{2}\right]\right\}^n$ qui satisfait le besoin.

  Pour l'estimation de la hauteur, $F(T_0,\ldots,T_n)$ et $F(T_0+a_0T_n,\ldots,T_{n-1}+a_{n-1}T_{n-1},T_n)$ sont ad\'eliquement primitifs. De plus, il y a au plus ${n+\delta\choose \delta}$ termes dans la somme \eqref{coefficient after automorphism}. En suite, une majoration de ${j_0\choose s_0}\cdots{j_{n-1}\choose s_{n-1}}$ dans \eqref{coefficient after automorphism} est ${\delta\choose\left[\frac{\delta+1}{2}\right]}$, et une majoration de $a_0^{j_0-s_0}\cdots a_{n-1}^{j_{n-1}-s_{n-1}}$ dans \eqref{coefficient after automorphism} est $\left(\frac{\delta+1}{2}\right)^\delta$. Donc on a la majoration de la hauteur de $h(\sigma(\mathscr X))$ dans l'\'enonc\'e.
\end{proof}
\subsection{Description num\'erique des fibres non r\'eduites}
Soit $\overline{\mathcal E}$ le fibr\'e vectoriel sur $\spec\O_K$ d\'efini dans \eqref{sE with l^2}. Dans la suite, on d\'esigne par $\mathbb P^n_K=\mathbb P(\sE_K)$ et $\mathbb P^n_{\O_K}=\mathbb P(\sE)$ pour simplifier. Soit $X\hookrightarrow\mathbb P^n_K$ l'hypersurface d\'efinie par un polyn\^ome homog\`ene
\begin{equation*}f(T_0,\ldots,T_n)=\sum_{\begin{subarray}{x}(i_0,\ldots,i_n)\in\mathbb N^{n+1}\\ i_0+\cdots+i_n=\delta\end{subarray}}a_{i_0,\ldots, i_n}T_0^{i_0}\cdots T_n^{i_n}\in K[T_0,\ldots,T_n]\end{equation*}
de degr\'e $\delta$, et $\mathscr X$ l'adh\'erence sch\'ematique de $X$ dans $\mathbb P^n_{\O_K}$. Soit $F\in\mathbb A_{\O_K}[T_0,\ldots,T_n]$ un polyn\^ome ad\'eliquement primitif associ\'e \`a $f$ comme construit dans \eqref{variable T_n adelic} de \S \ref{adelicly primitive}. Pour chaque id\'eal maximal $\p\in\spm\O_K$, la r\'eduction de $\mathscr X$ modulo $\p$ se factorise par la localisation de $\O_K$ vers $\O_{K,\p}$. En effet, on a le diagramme cart\'esien
   \[\xymatrix{\mathscr X_{\f_\p}\ar[r]\ar@{^{(}->}[d] \ar@{}[rd]|{\Box}& \mathscr X_{\O_{K,\p}} \ar[r]\ar@{^{(}->}[d]\ar@{}[rd]|{\Box}& \mathscr X\ar@{^{(}->}[d]\\\mathbb P^n_{\f_\p}\ar[r]\ar[d] \ar@{}[rd]|{\Box}& \mathbb P^n_{\O_{K,\p}} \ar[r]\ar[d]\ar@{}[rd]|{\Box}& \mathbb P_{\O_K}^n\ar[d]\\ \spec \f_\p\ar[r]&\spec\O_{K,\p} \ar[r]&\spec\O_K.}\]
   Par d\'efinition, $\mathscr X_{\O_{K,\p}}\hookrightarrow\mathbb P^n_{\O_{K,\p}}$ est d\'efini par la $\p$-partie $F^{(\p)}(T_0,\ldots,T_n)$ (voir la d\'efinition \ref{p-part}) de $F(T_0,\ldots,T_n)$ au-dessus, qui est primitif sur $\O_{K,\p}$.

   Avec les r\'esultats ci-dessus, on va d\'emontrer le r\'esultat suivant.
\begin{theo}\label{default reduissant}
  Soient $X\hookrightarrow\mathbb P^n_K$ une hypersurface r\'eduite, et $\mathscr X$ l'adh\'erence sch\'ematique de $X$ dans $\mathbb P^n_{\O_K}$. Soit $N(\p)=\#(\O_K/\p)$, o\`u $\p\in \spm\O_K$. Avec les notations au-dessus, on a l'in\'egalit\'e
  \[\frac{1}{[K:\Q]}\sum_{\p\in \mathcal Q(\mathscr X)}\log N(\p)\leqslant(2\delta-1)h(X)+C_1(n,\delta),\]
o\`u la constante
\begin{eqnarray*}C_1(n,\delta)&=&(2\delta-1)\delta\log\left(\frac{\delta+1}{2}\right)+\log(2\delta^{\delta}-\delta^{\delta-1})+(3\delta-2)\log {n+\delta\choose n}\\
& &+(2\delta-1)\log{\delta\choose\left[\frac{\delta+1}{2}\right]},\end{eqnarray*}
la notation $\mathcal Q(\mathscr X)$ est dans \eqref{Q(W)}, et $h(X)$ est d\'efinie dans la d\'efinition \ref{classical height of hypersurface}.
\end{theo}
\begin{proof}
Si $\delta=1$, alors $X$ est un hyperplan dans $\mathbb P_K^n$. Dans ce cas-l\`a, on a
\[\frac{1}{[K:\Q]}\sum_{\p\in \mathcal Q(\mathscr X)}\log N(\p)=0\]
par la d\'efinition directement, qui satisfait l'assertion.

Dans la suite on suppose $\delta\geqslant2$. On choisit un \'el\'ement $\sigma \in \Aut_{\O_K}\left(\mathbb P^n_{\O_K}\right)=\mathrm{PGL}_{n}(\O_K)$, qui envoie la coordonn\'ee $T_i$ dans $T_i+a_iT_n$ pour les $i=0,\ldots,n-1$ o\`u $a_i\in\mathbb Z$, et envoie $T_n$ dans $T_n$. Par le lemme \ref{non-reduced invariant}, on a $\mathcal Q(\mathscr X)=\mathcal Q(\sigma(\mathscr X))$. De plus, d'apr\`es le lemme \ref{change of coordinate} et le calcul dans \eqref{coefficient after automorphism}, il existe un $\sigma\in\Aut_{\O_K}\left(\mathbb P^n_{\O_K}\right)$, tel que
\[h(\sigma(\mathscr X))\leqslant h(X)+\log\left({n+\delta\choose\delta}{\delta\choose\left[\frac{\delta+1}{2}\right]}\left(\frac{\delta+1}{2}\right)^\delta\right)\]
 et la fibre g\'en\'erique de $\sigma(\mathscr X)$ soit d\'efinie par un polyn\^ome homog\`ene dont le coefficient du terme $T_n^\delta$ ne soit pas z\'ero.

On suppose la fibre g\'en\'erique de $\sigma(\mathscr X)\hookrightarrow\mathbb P^n_{\O_K}$ est d\'efinie par le polyn\^ome homog\`ene $f(T_0,\ldots,T_n)$ \`a coefficients dans $K$, et on \'ecrit
\begin{equation}\label{variable T_n 2}
    f(T_0,\ldots,T_n)=t_{\delta}T_n^{\delta}+t_{\delta-1}(T_0,\ldots,T_{n-1})T_n^{\delta-1}+\cdots+t_{0}(T_0,\ldots,T_{n-1}),
  \end{equation}
o\`u $t_\delta\neq0$. D'apr\`es la forme \eqref{variable T_n 2}, on \'ecrit le polyn\^ome ad\'eliquement primitif $F(T_0,\ldots,T_n)\in\mathbb A_{\O_K}[T_0,\ldots,T_0]$ de $f$ obtenu dans \S \ref{adelicly primitive} sous la forme de
  \[F(T_0,\ldots,T_n)=t'_{\delta}T_n^{\delta}+t'_{\delta-1}(T_0,\ldots,T_{n-1})T_n^{\delta-1}+\cdots+t'_{0}(T_0,\ldots,T_{n-1}),\]
  o\`u tous les $t'_i$ sont obtenus par la multiplication de $t_i$ par $c\in\mathbb A_{K}$ dans \eqref{variable T_n 2} pour les $i=0,\ldots,\delta$.

On consid\`ere $F(T_0,\ldots,T_n)$ comme un polyn\^ome de variable $T_n$ sur l'anneau $\mathbb A_{\O_K}[T_0,\ldots,T_{n-1}]$ de degr\'e $\delta$. Comme $X$ est r\'eduit, alors $F(T_0,\ldots,T_n)$ n'a pas de facteur carr\'e d'apr\`es la remarque \ref{projective hypersurface lost its reduceness}. Donc pour tout $\p\in\spm\O_K$, on a $\res'_{\delta}\left(F,\frac{\partial F}{\partial T_n}\right)^{(\p)}\neq0$ (voir \eqref{def of res'} pour la d\'efinition de ce r\'esultant). Donc si $F^{(\p)}(T_0,\ldots,T_n)$ modulo $\p[T_0,\ldots,T_{n}]$ admet un facteur carr\'e \`a la variable $T_n$, le polyn\^ome $\res'_{\delta}\left(F,\frac{\partial F}{\partial T_n}\right)^{(\p)}$ modulo $\p[T_0,\ldots,T_{n-1}]$ s'annule.

Si $F(T_0,\ldots,T_n)$ modulo $\p[T_0,\ldots,T_n]$ a un facteur carr\'e sans variable $T_n$, on a $\p\in\mathcal P(t'_{\delta})=\{\p\in\spm\O_K|{t'}_{\delta}^{(\p)}\mod \p=0\}$. Donc on a
\[\mathcal Q(\sigma(\mathscr X))\subseteq\mathcal P(t'_\delta)\cup\mathcal P\left(\res'_{\delta}\left(F,\frac{\partial F}{\partial T_n}\right)\right),\]
voir \eqref{R(F)} pour la d\'efinition de $\mathcal P\left(\res'_{\delta}\left(F,\frac{\partial F}{\partial T_n}\right)\right)$.

Pour l'estimation de $\mathcal P(t'_\delta)$, on a
\begin{eqnarray*}
\frac{1}{[K:\Q]}\sum_{\p\in\mathcal P\left(t'_{\delta}\right)} \log N(\p)&\leqslant&-\sum_{\p\in\spm\O_K}\frac{[K_\p:\Q_\p]}{[K:\Q]}\log|t'_{\delta}|_\p\\
&=&-\sum_{v\in M_{K}}\frac{[K_v:\Q_v]}{[K:\Q]}\log|ct_{\delta}|_v+\sum_{v\in M_{K,\infty}}\frac{[K_v:\Q_v]}{[K:\Q]}\log|t'_{\delta}|_v\\
&\leqslant&h_\infty(F)= h(\sigma(F),
\end{eqnarray*}
o\`u la ligne derni\`ere ci-dessus est d'apr\`es \eqref{h_infty=h} car $F$ est ad\'eliquement primitif et $t'_\delta$ est un coefficient de $F$, voir la d\'efinition \ref{adelic infinite height} pour la d\'efinition de $h_{\infty}(F)$.

Par la relation ci-dessus, on a
\begin{eqnarray*}
  & &\frac{1}{[K:\Q]}\sum_{\p\in \mathcal Q(\sigma(\mathscr X))}\log N(\p)\\
  &\leqslant&\frac{1}{[K:\Q]}\sum_{\p\in\mathcal P\left(\res'_{\delta}\left(F,\frac{\partial F}{\partial T_n}\right)\right)} \log N(\p)+\frac{1}{[K:\Q]}\sum_{\p\in\mathcal P\left(t'_{\delta}\right)} \log N(\p)\\
&\leqslant&(2\delta-2)h(\sigma(\mathscr X))+\log (2\delta^{\delta}-\delta^{\delta-1})+(\delta-1)\log {\delta+n\choose n}+h(\sigma(\mathscr X))\\
&\leqslant&(2\delta-1)h(X)+\log (2\delta^{\delta}-\delta^{\delta-1})+(\delta-1)\log {\delta+n\choose n}\\
& &+(2\delta-1)\log\left({n+\delta\choose\delta}{\delta\choose\left[\frac{\delta+1}{2}\right]}\left(\frac{\delta+1}{2}\right)^\delta\right),
\end{eqnarray*}
o\`u la deuxi\`eme in\'egalit\'e ci-dessus est d\'eduite \`a partir de la proposition \ref{r\'esultant2}, et la derni\`ere in\'egalit\'e est du lemme \ref{change of coordinate}. Donc on a l'assertion.
\end{proof}
\section{Contr\^ole des fibres non r\'eduites d'un sch\'ema de dimension pure}
Dans cette section ,on fixe $\overline{\mathcal E}$ le fibr\'e vectoriel sur $\spec\O_K$ d\'efini dans \eqref{sE with l^2}, et on d\'esigne $\mathbb P^n_K=\mathbb P(\sE_K)$ et $\mathbb P^n_{\O_K}=\mathbb P(\sE)$ pour simplifier. Soit $X$ un sous-sch\'ema ferm\'e de dimension pure de $\mathbb P^n_K$, qui est de dimension $d$ et degr\'e $\delta$. On d\'esigne par $\mathscr X$ l'adh\'erence sch\'ematique de $X$ dans $\mathbb P^n_{\O_K}$. Dans cette section, on contr\^olera les r\'eductions non r\'eduites de $\mathscr X$ lorsque $X$ est r\'eduit.

D'apr\`es la proposition \ref{compatible with flat base change}, la formation du diviseur de Cayley de $\mathscr X\hookrightarrow\mathbb P^n_{\O_K}$ commute au changement de base de $\O_K$ vers un corps r\'esiduel, voir \S \ref{Cayley form over field} pour la formation sur un corps, et \S \ref{formation of Cayley form over dedekind ring} pour la formation sur $\spec\O_K$. Donc pour c\^ontroler les fibres non r\'eduites de $\mathscr X\rightarrow\spec\O_K$, d'apr\`es la proposition \ref{cayleyform}, on a besoin de consid\'erer les r\'eduction du diviseur de Cayley de $\mathscr X$ telles que les vari\'et\'es de Cayley correspondantes ne soient pas r\'eduites.

Par l'argument ci-dessus, on a le c\^ontrole des fibres non r\'eduites de $\mathscr X$ suivant.

\begin{theo}\label{reduced default}
  Soient $X$ un sous-sch\'ema ferm\'e r\'eduit de dimension pure de $\mathbb P^n_K$, qui est de dimension $d$ et de degr\'e $\delta$, et $\mathscr X$ l'adh\'erence sch\'ematique $X$ dans $\mathbb P^n_{\O_K}$. De plus, soient $N={n+1\choose d+1}-1$, $\mathcal H_N=1+\frac{1}{2}+\cdots+\frac{1}{N}$, et $\overline{\O(1)}$ muni des m\'etriques de Fubini-Study induites par le fibr\'e vectoriel $\overline{\sE}$ d\'efini \`a \eqref{sE with l^2}, et $h_{\overline{\O(1)}}(\mathscr X)$ est la hauteur arakelovienne de $X$ d\'efinie dans la d\'efinition \ref{arakelov height of projective variety}. Alors on a
\[\frac{1}{[K:\Q]}\sum_{\p\in\mathcal Q(\mathscr X)}\log N(\p)\leqslant(2\delta-1)h_{\overline{\O(1)}}(\mathscr X)+C_2(n,d,\delta),\]
o\`u $\mathcal Q(\mathscr X)$ est d\'efini dans \eqref{Q(W)}, et la constante
\begin{eqnarray*}
  C_2(n,d,\delta)&=&(3\delta-2)\log {N+\delta\choose N}+(2\delta-1)\delta\log\frac{\delta+1}{2}+(2\delta-1)\log{\delta\choose\left[\frac{\delta+1}{2}\right]}\\
  & &+\log(2\delta^\delta-\delta^{\delta-1})+(2\delta-1)\left(4\delta\log(N+1)+(N+1)\delta\log 2-\frac{1}{2}\delta\mathcal H_N\right).
\end{eqnarray*}
\end{theo}
\begin{proof}
Soit $\Psi_{\mathscr X}$ le diviseur de Cayley de $\mathscr X$ prolong\'e dans $\mathbb P^n_{\O_K}$. Car $\mathscr X\rightarrow\spec\O_K$ est plat, donc $\Psi_{\mathscr X}$ est plat aussi sur $\spec\O_K$ par la proposition \ref{flatness of cayley form}. Pour tout $\p\in\spm\O_K$, soit $\Psi_{\mathscr X,\f_\p}$ la vari\'et\'e de Cayley de $\mathscr X_{\f_\p}\hookrightarrow\mathbb P^n_{\f_\p}$ d\'efinie dans la d\'efinition \ref{definition_of_Cayley_vareity} \`a partir de $\mathscr X\hookrightarrow\mathbb P^n_{\O_K}$ par la r\'eduction \`a $\p$. Alors d'apr\`es la proposition \ref{compatible with flat base change}, on a
\[\Psi_{\mathscr X}\times_{\spec\O_K}\spec\f_\p=[\Psi_{\mathscr X,\f_\p}]\]
comme diviseurs sur $\mathbb P\left(\bigwedge^{d+1}\sE_{\f_\p}\right)$. Donc $\Psi_{\mathscr X}\times_{\spec\O_K}\spec\f_\p$ d\'etermine la m\^eme hypersurface que $\Psi_{\mathscr X,\f_\p}$ dans $\mathbb P\left(\bigwedge^{d+1}\sE_{\f_\p}\right)$.

On d\'esigne
  \[\mathcal Q(\Psi_{\mathscr X})=\{\p\in\spm\O_K|\;\Psi_{\mathscr X,\f_\p}\hbox{ ne soit pas r\'eduit}\}.\]
Alors d'apr\`es la proposition \ref{compatible with flat base change}, la proposition \ref{cayleyform} et la remarque \ref{projective hypersurface lost its reduceness}, l'id\'eal maximal $\p\in\mathcal Q(\Psi_{\mathscr X})$ si et seulement si le sch\'ema $\mathscr X\times_{\spec\O_K}\spec \f_\p$ n'est pas r\'eduit, d'o\`u l'on a
  \begin{equation}\label{pure dimensional->cayley form}
    \frac{1}{[K:\Q]}\sum_{\p\in\mathcal Q(\mathscr X)}\log N(\p)=\frac{1}{[K:\Q]}\sum_{\p\in\mathcal Q(\Psi_{\mathscr X})}\log N(\p),
  \end{equation}
voir les notations ci-dessus dans \eqref{Q(W)}.

  D'apr\`es la proposition \ref{compatible with flat base change} aussi, le cycle $\Psi_{\mathscr X}\times_{\spec\O_K}\spec K$ d\'etermine une hypersurface de $\mathbb P\left(\bigwedge^{d+1}\sE_{K}\right)$ de degr\'e $\delta$. On prend un \'el\'ement $\psi_{X}\in \sym_{K}^\delta\left(\bigwedge^{d+1}\sE_K\right)$ qui d\'efinit l'hypersurface ci-dessus. Par le th\'eor\`eme \ref{default reduissant}, on obtient
   \begin{eqnarray}\label{inequality cayley form}
  & &\frac{1}{[K:\Q]}\sum_{\p\in\mathcal Q(\Psi_{\mathscr X})}N(\p)\\
  &\leqslant&(2\delta-1)h(\psi_{X})+\log(2\delta^\delta-\delta^{\delta-1})+(3\delta-2)\log{N+\delta\choose N}\nonumber\\
  & &+(2\delta-1)\delta\log\frac{\delta+1}{2}+(2\delta-1)\log{\delta\choose\left[\frac{\delta+1}{2}\right]},\nonumber
  \end{eqnarray}
o\`u $h(\psi_{X})$ est la hauteur classique d\'efinie dans la d\'efinition \ref{classical height of hypersurface}. On compare $h(\psi_{X})$ et $h_{\overline{\O(1)}}(\mathscr X)$ dans la proposition \ref{comparing heights} pour $\overline{\sE}$ d\'efini \`a \eqref{sE with l^2}, et on obtient l'assertion en l'appliquant \`a \eqref{pure dimensional->cayley form} et \eqref{inequality cayley form}.
\end{proof}
%\begin{coro}\label{reduced default2}
 % Avec toutes les notations dans le th\'eor\`eme \ref{reduced default}, soient $X$ un sous-sch\'ema ferm\'e r\'eduit de dimension pure de $\mathbb P^n_K$, qui est de dimension $d$ et degr\'e $\delta$, et $\mathscr X$ l'adh\'erence de Zariski de $X$ dans $\mathbb P^n_{\O_K}$. Pour un nombre r\'eel $N_0\geqslant2$, une majoration du nombre des places $\p$ avec $N_\p>N_0$ telles que $\mathscr{X}\times_{\spec \O_K}\spec \f_\p$ ne soit pas r\'eduit est
%\begin{equation*}
 % \frac{[K:\Q]}{\log N_0}\Big((2\delta-2)h_{\overline{\mathcal L}}(X)+C_2(n,d,\delta)\Big),
%\end{equation*}
 % o\`u la constante $C_2(n,d,\delta)$ est d\'efinie dans le th\'eor\`eme \ref{reduced default}, $N={n+1\choose d+1}-1$, $\mathcal H_N=1+\frac{1}{2}+\cdots+\frac{1}{N}$, et $\overline{\mathcal L}=\overline{\O}_X(1)$ muni des normes induites par des normes sur $\overline{\sE}$, et $h_{\overline{\mathcal L}}(X)$ est la hauteur arakelovienne de $X$.
%\end{coro}
%\begin{proof}
 % C'est une cons\'equence directe du th\'eor\`eme \ref{reduced default} et la proposition \ref{upper2}, o\`u la m\'ethode est m\^eme comme celle du corollaire \ref{default reduissant2}.
%\end{proof}
\begin{rema}
On consid\`ere la constante $C_2(n,d,\delta)$ d\'efinie dans le th\'eor\`eme \ref{reduced default}. On a\[C_2(n,d,\delta)\ll_{d,n}\delta^2\log\delta.\]
\end{rema}

\backmatter

\bibliography{liu}
\bibliographystyle{smfplain}

\end{document}